\documentclass[a4paper,12pt]{amsart}

\usepackage{amsmath}
\usepackage{amssymb}
\usepackage{amsthm}
\usepackage{verbatim}

\newtheorem{Theorem}{Theorem}[section]
\newtheorem{Proposition}[Theorem]{Proposition}
\newtheorem{Corollary}[Theorem]{Corollary}
\newtheorem{Lemma}[Theorem]{Lemma}

\theoremstyle{definition}

\newtheorem{Remark}[Theorem]{Remark}
\newtheorem{nt}[Theorem]{Notation}

\numberwithin{equation}{section}
\newcommand{\el}[3]{(\iota \otimes \omega_{#2, #3})(#1)}

\newcommand{\npsi}{\mathcal{N}_{\psi}}
\newcommand{\nphi}{\mathcal{N}_{\varphi}}
\newcommand{\hpsi}{H_{\psi}}
\newcommand{\hphi}{H_{\varphi}}
\newcommand{\pipsi}{\pi_{\psi}}
\newcommand{\piphi}{\pi_{\varphi}}

\newcommand{\mphi}{\mathcal{M}_{\varphi}}

\newcommand{\Dom}{\mathcal{D}}
\newcommand{\Domt}{\mathcal{D}^\otimes}
\newcommand{\NHalf}{\nabla^{\frac{1}{2}}}

\newcommand{\conj}{S}

\newcommand{\djnot}{D_{J_0}}
\newcommand{\dnablanot}{D_{\nabla^\frac{1}{2}_0}}

\newcommand{\ICM}{{\rm IC}(M)}
\newcommand{\Q}{\mathcal{Q}}
\newcommand{\HS}{\textrm{B}_2}
\newcommand{\I}{\mathcal{I}}

\newcommand{\NHalfAcc}{\nabla'^{\frac{1}{2}}}
\newcommand{\IR}{\textrm{IR}}
\newcommand{\IC}{{\rm IC}}
\newcommand{\sgn}{\textrm{sgn}}
\newcommand{\ep}{\varepsilon}
\newcommand{\ot}{\otimes}
\newcommand{\vp}{\varphi}
\def\De{\Delta}
\def\om{\omega}
\def\cM{\mathcal{M}}
\def\La{\Lambda}
\long\def\symbolfootnotemark[#1]{\begingroup%

\def\thefootnote{\fnsymbol{footnote}}\footnotemark[#1]\endgroup} 
\long\def\symbolfootnotetext[#1]#2{\begingroup%

\def\thefootnote{\fnsymbol{footnote}}\footnotetext[#1]{#2}\endgroup} 
\begin{document}
\title[Modular properties of matrix coefficients]{Modular properties of matrix coefficients of corepresentations of a locally compact quantum group}
\author{ Martijn Caspers, Erik Koelink
}

\address{Radboud Universiteit Nijmegen, IMAPP, FNWI, Heyendaalseweg 135, 6525 AJ Nijmegen,
the Netherlands}
\email{caspers@math.ru.nl, e.koelink@math.ru.nl}

\date{ \noindent April 11, 2011 
}

% Insert title

\begin{abstract}
We give a formula for the modular operator and modular conjugation in terms of matrix coefficients of corepresentations of a quantum group in the sense of Kustermans and Vaes. As a consequence, the modular autmorphism group of a unimodular quantum group can be expressed in terms of matrix coefficients. As an application, we determine the Duflo-Moore operators for the quantum group analogue of the normaliser of $SU(1,1)$ in $SL(2,\mathbb{C}$).
\end{abstract}

\maketitle

\section{Introduction}\label{SectIntroduction}

The definition of locally compact quantum groups has been given by Kustermans and Vaes
\cite{KusV}, \cite{KusVII} at the turn of the millenium, and we use their definition
of locally compact quantum groups in this paper. We stick mainly to the von Neumann algebraic
setting \cite{KusVII}. Since the introduction of quantum groups in the 1980ies and their theoretical
development, many results known in the theory of groups have been generalised to
quantum groups in some setting. In particular, the theory of compact quantum groups has
been settled satisfactorily by Woronowicz establishing analogues of the Haar measure and
the Schur orthogonality relations for matrix elements of corepresentations analogous to
the group case, see \cite{Tim} and references given there.
In particular, in the Kustermans-Vaes approach to locally quantum groups there is a
well-defined notion of dual locally compact quantum group. Moreover, the double dual
gives back the original locally compact quantum group. 

In his thesis \cite[\S 3.2]{Des} Desmedt generalises the Plancherel theorem for locally compact groups to the setting of quantum groups. Imposing sufficient conditions on a quantum group reminiscent of the conditions of the classical Plancherel theorem, he proves a decomposition of the biregular corepresentation in terms of tensor products of irreducible corepresentations. The intertwining operator, also called the Plancherel transformation, is given in terms of fields of positive self-adjoint operators which correspond to classical Duflo-Moore operators. One consequence of the quantum Plancherel theorem is the existence of orthogonality relations of matrix coefficients in terms of these operators.  

\vspace{0.3cm}

The present paper focusses on the modular properties of matrix coefficients of a locally compact quantum group that satisfies the assumptions of Desmedt's Plancherel theorem. The orthogonality relations
suggest that modular properties of integrals of the matrix coefficients of corepresentations
of a locally compact quantum can be expressed in terms of the corresponding operators
of Duflo-Moore type. Here, we give the polar decomposition of the second operator \eqref{EqnConjugationI} as in the Tomita-Takesaki theorem
 for a general locally compact quantum group satisfying the conditions
of the Plancherel theorem, see Theorem \ref{ThmPlancherelLeft}. In the case of a unimodular locally compact
quantum group, we obtain an
explicit expression for the action of the modular automorphism group on
matrix elements of corepresentations. This result is presented in Theorem \ref{ThmModularExpression}.

In the second part of this paper, we determine the modular conjugation and the modular automorphism group for the
case of the locally compact quantum group associated with the normaliser of
$SU(1,1)$ in $SL(2,\mathbb{C})$. This quantum group was introduced in \cite{KoeKus} and further studied
in \cite{GrKoeKus}, where the explicit decomposition of the left regular corepresentation
is presented. We calculate the Duflo-Moore operators for almost all corepresentations
in the decomposition of the left regular corepresentation. This extends Desmedt's result in \cite[\S 3.5]{Des}, where he determines Duflo-Moore operators for the discrete series corepresentations using summation formulas for basic hypergeometric series instead of the modular formula obtained in the present paper. 

\vspace{0.3cm}

This paper is structured as follows. After introducing the notational conventions, we recall Desmedt's Plancherel theorem in Section \ref{SectPlancherel}. We indicate how his theorem implies orthogonality relations between matrix coefficients and prove a result about integrals of matrix coefficients that are square integrable, see Theorem \ref{ThmDomain}.  Next, in Section \ref{SectModular} we give a formula of the modular automorphism group of a unimodular quantum group in terms of matrix coefficients.  In Section \ref{SectExample} we apply the theory of Sections \ref{SectPlancherel} and \ref{SectModular} to determine the Duflo-Moore operators of the normaliser of $SU(1,1)$ in $SL(2, \mathbb{C})$. Appendix \ref{AppendixA} contains a technical result on direct integration and Appendix \ref{AppendixB} proves that the example of Section \ref{SectExample} satisfies the assumptions of the Plancherel theorem.

\section{Conventions and notation}

For results on weight theory on von Neumann algebras our main reference is \cite{TakII}. If $\varphi$ is a weight on a von Neumann algebra $M$, we use the  notation $\nphi = \left\{ x \in M \mid \varphi(x^\ast x) < \infty \right\}$ and $\mphi = \nphi^\ast \nphi$, $\mphi^+ = \mphi \cap M^+$. $\sigma_t^\varphi$ denotes the modular automorphism group of $\varphi$.

The definition of a locally compact quantum group we use is the one by Kustermans and Vaes \cite{KusV},  \cite{KusVII}. We briefly recall their notational conventions, see also \cite{KusLec}, \cite{Tim}.
Let $(M, \Delta)$ be a locally compact quantum group, where $M$ denotes the von Neumann algebra and $\Delta$ the comultiplication. So $\Delta$ is
normal $\ast$-ho\-mo\-mor\-phism $\Delta\colon M\to M\otimes M$ satisfying
$(\Delta\ot\iota)\Delta=(\iota\ot\Delta)\Delta$, where $\iota$ denotes the identity.
Moreover, there exist two normal semi-finite faithful weights
$\vp$, $\psi$ on $M$ so that
\[
\begin{split}
\vp\bigl((\om\ot\iota )\Delta(x)\bigr)\, &=\, \vp(x)\om(1), \qquad
\forall \ \om\in M^+_*,\, \forall\ x\in \cM^+_\vp
\qquad \text{(left invariance),}\\
\psi\bigl((\iota \ot\om)\Delta(x)\bigr)\, &=\, \psi(x)\om(1), \qquad
\forall \ \om\in M^+_*,\, \forall\ x\in \cM^+_\psi
\qquad \text{(right invariance)}.
\end{split}
\]
$\vp$ is the left Haar weight and $\psi$ the right Haar weight. $(\hphi, \Lambda, \piphi)$ and $(\hpsi, \Gamma, \pipsi)$ denote the GNS-constructions with respect to the left Haar weight $\varphi$ and the right Haar weight $\psi$ respectively. Without loss of generality we may
assume that $\hphi=\hpsi$ and $M\subset B(\hphi)$. The operator $W\in B(\hphi\ot\hphi)$ defined by
$W^\ast \bigl( \La(a)\ot \La(b) \bigr) =
\bigl( \La\ot\La\bigr) \bigl( \De(b)(a\ot 1)\bigr)$
 is a unitary operator known as the multiplicative unitary. It implements the
comultiplication
$\De(x)= W^\ast(1\ot x)W$ for all $x\in M$ and satisfies the pentagonal equation
$W_{12}W_{13}W_{23}=W_{23}W_{12}$ in $B(\hphi\ot \hphi\ot\hphi)$. 
In \cite{KusV}, \cite{KusVII}, see also \cite{KusLec}, \cite{Tim}, it is proved that there exists
a dual locally compact quantum group $(\hat{M},\hat{\Delta})$, so that
$(\hat{\hat{M}},\hat{\hat{\Delta}}) = (M,\Delta)$.

A unitary corepresentation
$U$ of a von Neumann algebraic quantum
group on a Hilbert space $H$ is a unitary element $U\in M\ot B(H)$
such that $(\De\ot\iota)(U)=U_{13}U_{23}\in M\ot M\ot B(H)$, where
the standard leg-numbering is used in the right hand side.
A closed subspace $L\subseteq H$ is an
invariant subspace for the unitary corepresentation $U$ 
if $(\om\ot \iota)(U)$ preserves $L$ for all $\om\in M_\ast$.
A unitary corepresentation $U$
in the Hilbert space $H$
is irreducible if there are only trivial (i.e. equal to
$\{0\}$ or the whole Hilbert space $H$)  invariant subspaces. If $U_1$ is a corepresentation on a Hilbert space $H_1$ and $U_2$ is a corepresentation on a Hilbert space $H_2$, then $U_1$ is equivalent to $U_2$ if there is a unitary map $\Upsilon: H_1 \rightarrow H_2$, such that $(\iota \otimes \Upsilon) U_1 = U_2 (\iota \otimes \Upsilon)$.
We use the notation $\IC(M)$ for the equivalence classes of irreducible, unitary corepresentations of $(M,\Delta)$.

$(\hat{M}_u, \hat{\Delta}_u)$ denotes the universal dual and $(\hat{M}_c, \hat{\Delta}_c)$ denotes the reduced dual C$^\ast$-algebraic quantum groups \cite{Kus}. The dual weights are denoted by $\hat{\varphi}_u$ and $\hat{\psi}_u$  for  $(\hat{M}_u, \hat{\Delta}_u)$ and $\hat{\varphi}_c$ and $\hat{\psi}_c$ for $(\hat{M}_c, \hat{\Delta}_c)$.   Similarly, we have GNS-constructions $(\hphi, \hat{\Lambda}_{\hat{\varphi}_u}, \pi_{\hat{\varphi}_u})$ and $(\hpsi, \hat{\Gamma}_{\hat{\psi}_u}, \pi_{\hat{\psi}_u})$ for $(\hat{M}_u, \hat{\Delta}_u)$ and $(\hphi, \hat{\Lambda}_{\hat{\varphi}_c}, \pi_{\hat{\varphi}_c})$ and $(\hpsi, \hat{\Gamma}_{\hat{\psi}_c}, \pi_{\hat{\psi}_c})$ for $(\hat{M}_c, \hat{\Delta}_c)$. Recall that without loss of generality we may assume that $\hphi$ equals $\hpsi$.

 By $\IR(\hat{M}_u)$ and $\IR(\hat{M}_c)$ we denote the equivalence classes of irreducible, unitary representations of $\hat{M}_u$ and $\hat{M}_c$ respectively. We recall from \cite{Kus} that there is a bijective correspondence between $\IR(\hat{M}_u)$ and $\IC(M)$ and that $\IR(\hat{M}_c)$ is contained in $\IR(\hat{M}_u)$.

$W$ denotes the multiplicative unitary associated with $M$.  For
$\omega \in M_\ast$ we define $\lambda(\omega) = (\omega\otimes
\iota)(W) \in \hat{M}$. We set
$$\mathcal{I} = \left\{\omega \in M_\ast \mid \Lambda(x) \mapsto \omega(x^\ast), x \in \nphi \textrm{ is a continuous functional on } \hphi \right\}.$$
\noindent $\mathcal{I}$ is dense in $M_\ast$ \cite[Lemma 8.5]{KusV}. By the Riesz representation theorem, for every $\omega \in \mathcal{I}$ one can associate a unique vector denoted by $\xi(\omega)$ such that $\langle \xi(\omega), \Lambda(x) \rangle = \omega(x^\ast)$. The set $\xi(\omega)$, $\omega \in \mathcal{I}$, is dense in $\hphi$ \cite[Lemma 8.5]{KusV}. Then the dual weight $\hat{\varphi}$ on $\hat{M}$ is the weight defined by the GNS-construction $\lambda(\omega) \mapsto \xi(\omega)$. This GNS-construction of $\hat{M}$ is denoted by $\hat{\Lambda}$. All these definitions have right analogues.
$$\mathcal{I}_R = \left\{\omega \in M_\ast \mid \Gamma(x) \mapsto \omega(x^\ast), x \in \npsi \textrm{ is a continuous functional on } \hpsi \right\}.$$
\noindent For $\omega \in \mathcal{I}_R$, there is a vector $\xi_R(\omega)$ such that $\langle \xi_R(\omega), \Gamma(x) \rangle = \omega(x^\ast)$. The set $\xi_R(\omega)$, $\omega \in \mathcal{I}_R$, is dense in $\hphi$. 

For $\alpha \in M_\ast$, we denote $\overline{\alpha} \in M_\ast$ for the functional defined by $\overline{\alpha}(x) = \overline{\alpha(x^\ast)}$. Define $M_\ast^\sharp = \left\{ \alpha \in M_\ast \mid \exists \theta \in M_\ast: (\theta \otimes \iota)(W) = (\alpha \otimes \iota)(W)^\ast \right\}.$ It can be shown \cite{Kus} that for every $\alpha \in M_\ast^\sharp$ there is a unique $\theta \in M_\ast$ such that $(\theta \otimes \iota)(W) = (\alpha \otimes \iota)(W)^\ast$ and $\theta$ is determined by $\theta(x) = \overline{\alpha}(\mathcal{S}(x)), x \in \Dom(\mathcal{S})$, where $\mathcal{S}$ is the unbounded antipode of $(M, \Delta)$. We will write $\alpha^\ast$ for this $\theta$.

Basic results on direct integration can be found in \cite{Dix}.
For direct integrals of unbounded operators we refer to
\cite{Lan}, \cite{Nus} and \cite[Chapter 12]{Schm}. If $X$ is a standard measure space with measure $\mu$,
we use the notation $(H_U)_{U\in X}$ or simply $(H_U)_U$ for a
field of of Hilbert spaces $H_U$ over $X$. If $(H_U)_U$ is a
measurable field of Hilbert spaces we denote its direct integral
by $\int^\oplus_{X} H_U d\mu(U)$. Similarly we add subscripts to denote fields of vectors,
operators and representations.

Let $H$ be a Hilbert space. We define the inner product to be linear in the first entry and anti-linear in the second entry. We denote the Hilbert-Schmidt operators on $H$ by $B_2(H)$. Recall that $B_2(H)$ is a Hilbert space itself, which is isomorphic to $H \otimes \overline{H}$, the isomorphism being given by $\xi \otimes \overline{\eta}: h \mapsto \langle h, \eta \rangle \xi$. Here $\overline{H}$ denotes the conjugate Hilbert space. We denote vectors in $\overline{H}$ and operators acting on $\overline{H}$ with a bar. For $\xi, \eta \in H$ the normal functional $\omega_{\xi, \eta}$ on $B(H)$ is defined as $\omega_{\xi,\eta}(A) = \langle A \xi, \eta\rangle$. The domain of an (unbounded) operator $A$ on $H$ is denoted by $\Dom(A)$. The symbol $\otimes$  denotes either the tensor product of Hilbert spaces, the tensor product of operators or the von Neumann algebraic tensor product. It will always be clear from the context which tensor product is meant.

\section{Plancherel Theorems}\label{SectPlancherel}

The classical Plancherel theorem for locally compact groups \cite[Theorem 18.8.1]{DixC} has a quantum group analogue, which has been proved by Desmedt in \cite{Des}. This section recalls part of Desmedt's Plancherel theorem and elaborates on minor modifications and implications of this theorem which turn out to be useful for explicit computations in Section \ref{SectExample}. 

For two unbounded operators $A$ and $B$, we denote $A \cdot B$ for the closure of their product.

\begin{Theorem}[{\rm Desmedt \cite[Theorem 3.4.1]{Des}}]\label{ThmPlancherelLeft}
Let $(M, \Delta)$ be a locally compact quantum group such that $\hat{M}$ is a type I von Neumann algebra and such that $\hat{M}_u$ is a separable C$^\ast$-algebra. There exist a standard
measure $\mu$ on $\ICM$, a measurable field $(H_U)_U$ of Hilbert spaces, a
measurable field $(D_U)_U$ of self-adjoint, strictly positive operators and an
isomorphism $\Q_L$ of
$\hphi$ onto $\int^\oplus \HS(H_U) d\mu(U)$ with the following properties:

\begin{enumerate}
\item For all $\alpha \in \I$ and $\mu$-almost all $U \in \ICM$, the operator $(\alpha \otimes \iota)(U)D_U^{-1}$ is bounded and $(\alpha \otimes \iota)(U) \cdot D_U^{-1}$ is a
Hilbert-Schmidt operator on $H_U$.
\item For all $\alpha, \beta \in \I$ one has the Parseval formula
$$\langle \xi(\alpha), \xi(\beta)\rangle = \int_{\ICM} \!\!{\rm Tr}\left(\left((\beta \otimes \iota)(U) \cdot D_U^{-1}\right)^\ast \left((\alpha \otimes \iota)(U) \cdot D_U^{-1}\right)\right) d\mu (U), $$
and $\Q_L$ is the isometric extension of
$$\hat{\Lambda}(\lambda(\I)) \rightarrow \int^{\oplus}_{\ICM} \HS(H_U) d\mu(U): \:\:\xi(\alpha) \mapsto \int^\oplus_{\ICM} (\alpha \otimes \iota)(U) \cdot D_U^{-1}d\mu(U).$$
\end{enumerate}
\end{Theorem}
%---------------------------------------------------
 
%---------------------------------------------------

Here $\mu$ is called the left Plancherel measure and $\Q_L$ is called the left Plancherel transform. We will be dealing with a right analogue of the Plancherel theorem as well, see \cite[Remark 3.4.11]{Des}. Here we explicitly state the part of this theorem that is relevant for the present paper. 

\begin{Theorem}\label{ThmPlancherelRight}
Let $(M, \Delta)$ be a locally compact quantum group such that
$\hat{M}$ is a type-I von Neumann algebra and such that
$\hat{M}_u$ is a separable C$^\ast$-algebra. There exist a standard measure $\mu_R$ on
$\ICM$, a measurable field $(K_U)_U$ of Hilbert spaces, a
measurable field $(E_U)_U$ of self-adjoint, strictly positive
operators and an isomorphism $\Q_R$ of $\hpsi$ onto $\int^\oplus
\HS(K_U) d\nu(U)$ with the following properties:

\begin{enumerate}
\item\label{ThmPlancherelRightI} For all $\alpha \in \I_R$ and $\mu_R$-almost all $U \in \ICM$, the operator $(\alpha \otimes \iota)(U)E_U^{-1}$ is bounded and $(\alpha \otimes \iota)(U)\cdot E_U^{-1}$ is a
Hilbert-Schmidt operator on $K_U$.
\item\label{ThmPlancherelRightII} For all $\alpha, \beta \in \I_R$ one has the Parseval formula
$$\langle \xi_R(\overline{\alpha^\ast}), \xi_R(\overline{\beta^\ast})\rangle = \int_{\ICM} \!\!\!\!{\rm Tr}\left(\left((\beta \otimes \iota)(U) \cdot E_U^{-1}\right)^\ast \left((\alpha \otimes \iota)(U) \cdot E_U^{-1}\right)\right) d\mu_R(U), $$
and $\Q_R$ is the isometric extension of
$$\xi_R(\overline{\mathcal{I}_R^\ast})  \rightarrow \int^{\oplus}_{\ICM} \HS(H_U) d\mu_R(U): \:\:\xi_R(\overline{\alpha^\ast}) \mapsto \int^\oplus_{\ICM} (\alpha \otimes \iota)(U) \cdot E_U^{-1}d\mu_R(U).$$
\item The measure $\mu_R$ can be choosen equal to the measure $\mu$ of Theorem \ref{ThmPlancherelLeft} and the measurable field of Hilbert spaces $(K_U)_U$ can be choosen equal to $( H_U)_U$, the measurable field of Hilbert spaces of Theorem \ref{ThmPlancherelLeft}.
\end{enumerate}
\end{Theorem}

Parts (\ref{ThmPlancherelRightI}) and (\ref{ThmPlancherelRightII}) of this theorem can be obtained from \cite[Theorem 3.4.5]{Des} using the relations between the right Haar weight $\psi$ and the right Haar weight $\hat{\psi}_u$ on the universal dual quantum group. The prove is similar to how Theorem \ref{ThmPlancherelLeft} is obtained from \cite[Theorem 3.4.5]{Des}.  We elaborate a bit on the third statement. Since $\hat{\varphi}_u$ and $\hat{\psi}_u$ are both approximately KMS-weights on the universal dual $\hat{M}_u$, their W*-lifts are n.s.f. weights so that \cite[Theorem VIII.3.2]{TakII} implies that the representations $\piphi$ and $\pipsi$ are equivalent. Hence
\begin{equation}\label{EqnEquivReps}
\pi_{\hat{\varphi}_u}(\hat{M}_u)'' = \pi_{\hat{\varphi}}(\hat{M}) \simeq \pi_{\hat{\psi}}(\hat{M}) = \pi_{\hat{\psi}_u}(\hat{M}_u)''.
\end{equation}
\noindent The proofs of Theorems \ref{ThmPlancherelLeft} and \ref{ThmPlancherelRight} show that the measures $\mu$ and $\nu$ together with the measurable fields of Hilbert spaces $(H_U)_U$ and $(K_U)_U$ in Theorems \ref{ThmPlancherelLeft} and \ref{ThmPlancherelRight} arise from the direct integral decompositions of $\pi_{\hat{\varphi}_u}(\hat{M}_u)''$ and $\pi_{\hat{\psi}_u}(\hat{M}_u)''$, respectively. That is:
$$\pi_{\hat{\varphi}_u}(\hat{M}_u)'' = \int^\oplus_{X} B(H_\sigma) d\mu(\sigma), \qquad \pi_{\hat{\psi}_u}(\hat{M}_u)'' = \int^\oplus_{Y} B(K_\sigma) d\mu_R(\sigma).$$
\noindent By (\ref{EqnEquivReps}) we may assume that $\mu = \mu_R$, $X=Y$  and $(H_U)_U = (K_U)_U$. Furthermore, by \cite[Eqn. (3.4.2)]{Des},  $\pi_{\hat{\varphi}}(y) = y = \pi_{\hat{\psi}}(y), \forall y \in \hat{M}$, which shows that the correspondence between $X$ and the measurable subspace $\IR(\hat{M}_u)$ is the same for $\pi_{\hat{\varphi}_u}$ and $\pi_{\hat{\psi}_u}$.  This proves the third statement of Theorem \ref{ThmPlancherelRight}.

\vspace{0.3cm}

$\Q_R$ is
called the right Plancherel transform. In the rest of this paper we will assume that $\mu = \mu_R$ and we simply call $\mu$ the Plancherel measure. Similarly, we
identify $(K_U)_U$ with $(H_U)_U$.

\begin{Remark}\label{RmkReducedSetting}
Theorems \ref{ThmPlancherelLeft} and \ref{ThmPlancherelRight} remain valid when the assumption that $\hat{M}_u$ is separable (universal norm) is replaced by the assumption that $\hat{M}_c$ is separable (reduced norm) and the measure space $\IC(M)$ is replaced by the measure space $\IR(\hat{M}_c)$. The proof is a minor modification of the proof of \cite[Theorem 3.4.1]{Des}. Here, $\IR(\hat{M}_u)$ can be replaced by $\IR(\hat{M}_c)$ and $\hat{\mathcal{V}}$ should be read as the multiplicative unitary $W$, see \cite{Kus} for the definition of $\hat{\mathcal{V}}$. The proof of this modification can be obtained by using the following relations instead of \cite[p. 118-119]{Des}:
\begin{equation}
\begin{split}
  \pi_\sigma\left((\omega \otimes \iota)(W)\right)& = (\omega \otimes \iota)(U_\sigma)  \textrm{ where } \sigma\in \IR(\hat{M}_c) \textrm{ corresponds to } U_\sigma \in \IC(M), \\
\xi(\omega) &= \hat{\Lambda}\left( (\omega \otimes \iota)(W) \right) = \hat{\Lambda}_{\hat{\varphi}_c} \left( (\omega \otimes \iota) (W) \right).
\end{split}
\end{equation}
\noindent In \cite[Theorem 3.4.8]{Des} Desmedt proves that the support of the left and right Plancherel measures equal $\IR(\hat{M}_c)$, which is in agreement with this observation.
\end{Remark}

\begin{Remark}\label{RmkSqIntCoreps}
 The corepresentations that appear as discrete mass points in the Plancherel measure correspond to the square integrable correpresentations in the sense of \cite[Definition~3.2]{BusMey} or the equivalent definition of left square integrable corepresentations as in \cite[Definition~3.2.29]{Des}. A proof of this can be found in \cite[Theorem~3.4.10]{Des}.
\end{Remark}

\begin{nt}\label{NtNotationDrieEnVier} In the rest of this section as well as in Section \ref{SectModular} we adopt the following notational conventions. $(M,\Delta)$ is a fixed locally compact quantum group satisfying the conditions of Theorems \ref{ThmPlancherelLeft} and \ref{ThmPlancherelRight}. We set $D = \int^\oplus_{\IC(M)} D_U d\mu(U)$, $E = \int^\oplus_{\IC(M)} E_U d\mu(U)$ and $H = \int^\oplus_{\IC(M)} H_U d\mu(U)$, where $\mu$ is the Plancherel measure. All (direct) integrals are taken over $\IC(M)$. In the proofs we omit this in the notation. 
\end{nt}

In the remainder of this Section, we express the Plancherel transformation in terms of matrix coefficients to arrive at Theorems \ref{ThmOrthogonalitRel} and \ref{ThmDomain}. These theorems can be considered as direct implications of the Plancherel theorems. We will need them in Section \ref{SectExample}. 

\begin{Lemma}\label{LemGNS}
We have the following:
\begin{enumerate}
\item Let $x\in M$, such that the linear map $f: \hat{\Lambda}(\lambda(\mathcal{I})) \rightarrow \mathbb{C}: \xi(\alpha) \mapsto \alpha(x^\ast)$ is bounded. Then $x \in
\Dom(\Lambda)$ and $f(\xi(\alpha)) = \langle \xi(\alpha), \Lambda(x)\rangle$.
\item Let $x\in M$, such that the linear map $f: \hat{\Gamma}(\lambda(\mathcal{I}_R)) \rightarrow \mathbb{C}: \xi_R(\alpha) \mapsto \alpha(x^\ast)$ is bounded. Then $x \in
\Dom(\Gamma)$ and $f(\xi_R(\alpha)) = \langle \xi_R(\alpha), \Gamma(x)\rangle$.
\end{enumerate}
\end{Lemma}
\begin{proof} We prove the first statement, the second being analogous.
The claim is true for $x \in \nphi$, since $\Dom(\Lambda) = \nphi$
and by definition $\langle \xi(\alpha), \Lambda(x) \rangle =
\alpha(x^\ast)$, for all $\alpha \in \mathcal{I}$.
Now, let $x \in M$ be arbitrary. The set $\{ \xi(\alpha) \mid \alpha \in \mathcal{I}
\}$ is dense in $\hphi$ by \cite[Lemma 8.5]{KusV} and its
subsequent remark. Hence, by the Riesz theorem, there is a $v \in
\hphi$ such that for every $\alpha \in \mathcal{I}$
$
\alpha(x^\ast) = \langle \xi(\alpha), v\rangle. 
$
 Let $(e_j)_{j \in J}$ be a bounded net in the Tomita algebra
\[
\mathcal{T}_\varphi = \left\{ x \in \nphi \cap \nphi^\ast \mid x \textrm{ is analytic w.r.t. } \sigma^\varphi \textrm{ and } \sigma_z^\varphi(x) \in \nphi \cap \nphi^\ast, \forall z \in \mathbb{C} \right\},
\]
 converging $\sigma$-weakly to 1 and such that $\sigma^\varphi_{i/2}(e_j)$ converges $\sigma$-weakly to 1 (using the residue formula for meromorphic functions, one can see that the net $(e_j)_{j \in J}$ defined in \cite[Lemma 9]{TerpII} satisfies these properties). Let $a, b \in \mathcal{T}_\varphi$ and fix the normal functional $\alpha$ by $\alpha(x) = \varphi(axb), x \in M$. Using \cite[Lemma 8.5]{KusV} we find
\[
\langle \xi(\alpha), \Lambda(x e_j) \rangle = \varphi(a e_j^\ast x^\ast b) = \langle \Lambda(b \sigma_{-i}^\varphi(a e_j^\ast) ), v \rangle = \langle \Lambda(b \sigma_{-i}^\varphi(a) ), J \piphi(\sigma_{i/2}^\varphi(e_j)^\ast) J v \rangle. 
\]
Hence, $\Lambda(x e_j) =  J \piphi(\sigma_{i/2}^\varphi(e_j)^\ast) J v $, so that $\Lambda(x e_j)$ converges weakly to $v$. Since $x e_j \rightarrow x$ $\sigma$-weakly and $\Lambda$ is $\sigma$-weak/weakly closed, this implies that $x \in \Dom(\Lambda) = \nphi$ and $v = \Lambda(x)$.
\end{proof}

Recall that $\int^\oplus B_2(H_U) d\mu(U) \simeq \int^\oplus H_U \otimes \overline{H_U} d\mu(U)$. For $\eta = \int^\oplus \eta_U d\mu(U), \xi =
\int^\oplus \xi_U d\mu(U) \in H$ the mesurable field of vectors $(\xi_U \otimes \overline{\eta_U})_U$ is not necessarily square integrable. If it is square integrable, $\int^\oplus \xi_U  \otimes \overline{\eta_U} d\mu(U) \in \int^\oplus B_2(H_U) d\mu(U)$.

We obtain the following expression for the left Plancherel transformation.

\begin{Lemma}\label{LemCompI}
Let $\eta = \int^\oplus_{\ICM}  \eta_U d\mu(U) \in H$ and $\xi =
\int^\oplus_{\ICM}  \xi_U d\mu(U) \in H$ be such that $\eta \in \Dom(D^{-1})$ and $(\xi_U  \otimes \overline{\eta_U})_U$ is square integrable.
Then $\ICM \ni U \mapsto (\iota \otimes
\omega_{\xi_U, D_U^{-1} \eta_U})(U^\ast) \in M$ is $\sigma$-weakly
integrable with respect to $\mu$ and $\int_{\IC(M)} (\iota
\otimes \omega_{\xi_U, D_U^{-1} \eta_U})(U^\ast) d\mu(U) \in
\nphi$, and
\begin{equation}\label{EqnPlancherelMatrix}
\Q_L^{-1}(\int^{\oplus}_{\ICM} \xi_U  \otimes \overline{\eta_U} d\mu(U)) = \Lambda\left( \int_{\ICM} (\iota \otimes \omega_{\xi_U, D_U^{-1} \eta_U})(U^\ast) d\mu(U)\right).
\end{equation}
\end{Lemma}
\begin{proof} For $\alpha \in \mathcal{I}$, Theorem
\ref{ThmPlancherelLeft} implies that
\[
\begin{split}
& \langle \xi(\alpha), \Q_L^{-1}(\int^{\oplus} \xi_U  \otimes \overline{\eta_U} d\mu(U)) \rangle  =  \int  \langle(\alpha \otimes \iota)(U)D_U^{-1}, \xi_U  \otimes \overline{\eta_U} \rangle_{{\rm HS}} d\mu(U) \\
 = &  \int (\alpha \otimes \omega_{D_U^{-1}\eta_U, \xi_U})(U) d\mu(U) =  \alpha (\int (\iota \otimes \omega_{\xi_U,
D_U^{-1}\eta_U})(U^\ast) d\mu(U)^\ast) ,
\end{split}
\]
where the last integral exists in the $\sigma$-weak sense. We see by Lemma \ref{LemGNS} that, $\int (\iota \otimes \omega_{\xi_U,
D_U^{-1} \eta_U})(U^\ast) d\mu(U) \in \Dom(\Lambda) = \nphi $,
and (\ref{EqnPlancherelMatrix}) follows.
\end{proof}

\begin{Remark}\label{RmkIntExists} As in the proof of Lemma \ref{LemCompI} we see that for $\xi = \int^\oplus \xi_U d\mu(U) \in H$, $\eta = \int^\oplus \eta_U d\mu(U) \in H$, the $\sigma$-weak integral $\int (\iota \otimes \omega_{\xi_U, \eta_U})(U^\ast) d\mu(U)\in M$ exists, and for $\alpha \in M_\ast$, $\vert \int (\alpha \otimes \omega_{\xi_U, \eta_U})(U) d\mu(U) \vert \leq \Vert \alpha \Vert \Vert \xi \Vert \Vert \eta \Vert$.
\end{Remark}

The previous lemma shows that $\Q_L^{-1}$ is an analogue of what Desmedt calls the (left) Wigner map \cite[Section 3.3.1]{Des}. This map is defined as
\begin{equation}
B_2(H_U) \rightarrow H:  \xi  \otimes \overline{\eta}  \mapsto \Lambda\left( (\iota \otimes \omega_{\xi, D_U^{-1} \eta})(U^\ast) \right),
\end{equation}
\noindent where $U$ is a corepresentation on a Hilbert space $H_U$ that appears as a discrete mass point in the Plancherel measure, cf. the remarks about square integrable corepresentations at the end of Section \ref{SectPlancherel}. This map is also considered in \cite[Page 203]{BusMey}, where it is denoted by $\Phi$.

The next Lemma is the right analogue of Lemma \ref{LemCompI}, the proof being similar.

\begin{Lemma}\label{LemCompII}
Let $\eta = \int^\oplus_{\ICM}  \eta_U d\mu(U) \in H$ and $\xi =
\int^\oplus_{\ICM}  \xi_U d\mu(U) \in H$ be such that $\eta \in \Dom(E^{-1})$ and $(  \xi_U  \otimes \overline{\eta_U})_U  $ is square integrable.
Then $\ICM \ni U \mapsto (\iota \otimes
\omega_{\xi_U, E_U^{-1} \eta_U})(U) \in M$ is $\sigma$-weakly integrable
with respect to $\mu$. Furthermore, $\int_{\ICM} (\iota \otimes
\omega_{\xi_U, E_U^{-1} \eta_U})(U) d\mu(U) \in \npsi$ and
$$\Q_R^{-1}(\int^{\oplus}_{\ICM}  \xi_U  \otimes \overline{\eta_U}  d\mu(U)) = \Gamma\left( \int_{\ICM} (\iota \otimes \omega_{\xi_U, E_U^{-1} \eta_U})(U) d\mu(U)\right).$$
\end{Lemma}

%Maybe a remark that actually $\int (\iota \otimes \omega_{\xi_U, \eta_U})(U)d\mu(U)$ always exists (in the weak-* sense) if $\xi_U$ and $\eta_U$ are sq. int.
 
The Plancherel theorems imply the following orthogonality relations.  The theorem follows immediately from the expressions for the Plancherel transformations given in Lemmas \ref{LemCompI} and \ref{LemCompII}. The orthogonality relations will be used in Section \ref{SectExample} where we give a method to determine the Duflo-Moore operators of a locally compact quantum group that satisfies the assumptions of the Plancherel theorem.

\begin{Theorem}[Orthogonality relations]\label{ThmOrthogonalitRel} Let $(M, \Delta)$ be a locally compact quantum group, such that $\hat{M}_u$ is separable and $\hat{M}$ is a type I von Neumann algebra. Let $\eta = \int^\oplus \eta_U d\mu(U) \in H$, $\xi = \int^\oplus
\xi_U d\mu(U) \in H$, $\eta' = \int^\oplus \eta_U' d\mu(U) \in H$
and $\xi' = \int^\oplus \xi_U' d\mu(U) \in H$. We have the
following orthogonality relations:
\begin{enumerate}
\item Suppose that $\eta, \eta' \in \Dom(D)$ and that $ (\xi_U  \otimes \overline{D_U \eta_U})_U ,
 (\xi_U '  \otimes \overline{D_U \eta_U '} )_U$ are square integrable fields of
vectors, then
\begin{equation}\label{EqnOrthI}
\begin{split}
&\varphi\left( \left(\int (\iota \otimes \omega_{\xi_U, \eta_U})(U^\ast) d\mu(U)\right) ^\ast \int (\iota \otimes \omega_{\xi_U', \eta_U'})(U^\ast) d\mu(U) \right) = \\
& \int \langle D_U \eta_U, D_U \eta'_U \rangle \langle \xi'_U, \xi_U
\rangle d\mu(U).
\end{split}
\end{equation}
\item Suppose that $\eta, \eta' \in \Dom(E)$ and that $ (\xi_U  \otimes \overline{E_U \eta_U})_U ,
 (\xi_U '  \otimes \overline{E_U \eta_U '})_U $ are square integrable fields of
vectors, then:
\begin{equation}\label{EqnOrthII}
\begin{split}
&\psi\left(\left( \int  (\iota \otimes \omega_{\xi_U, \eta_U})(U) d\mu(U)\right) ^\ast\int   (\iota \otimes \omega_{\xi_U', \eta_U'})(U) d\mu(U) \right) = \\
& \int  \langle E_U \eta_U, E_U \eta'_U \rangle \langle \xi'_U, \xi_U
\rangle d\mu(U).
\end{split}
\end{equation}
\end{enumerate}

\noindent Here $\int (\iota \otimes \omega_{\xi_U, \eta_U})(U) d\mu(U), \int (\iota \otimes \omega_{\xi_U, \eta_U})(U^\ast) d\mu(U), \int (\iota \otimes \omega_{\xi_U', \eta_U'})(U) d\mu(U)$,\\
\noindent  $\int (\iota \otimes \omega_{\xi_U', \eta_U'})(U^\ast) d\mu(U)$ are defined in Lemma \ref{LemCompI} and \ref{LemCompII}. The integrals are taken over $\ICM$.
\end{Theorem} 

 As observed in Remark \ref{RmkIntExists} the element  $\int (\iota \otimes \omega_{\xi_U,\eta_U})(U^\ast)d\mu(U) \in M$ exists for $\eta = \int^\oplus  \eta_U d\mu(U) \in H$, $\xi = \int^\oplus 
\xi_U d\mu(U) \in H$ and the next theorem investigates the consequences of $\int (\iota \otimes \omega_{\xi_U, \eta_U})(U^\ast)d\mu(U) \in \nphi$.

\begin{Theorem}\label{ThmDomain}
Let $\xi = \int^\oplus_{\ICM}  \xi_U d\mu(U) \in H$ be an essentially bounded field of vectors.
\begin{enumerate}
\item  Let $\eta = \int^\oplus_{\ICM}  \eta_U d\mu(U) \in H$ be such that $\int_{\ICM}  (\iota \otimes \omega_{\xi_U, \eta_U})(U^\ast) d\mu(U) \in \nphi$. Then, for almost every $U$ in the support of $(\xi_U)_U$, we have $\eta_U \in \Dom(D_U)$.
\item Let $\eta = \int^\oplus_{\ICM}  \eta_U d\mu(U) \in H$ be such that $\int_{\ICM}  (\iota \otimes \omega_{\xi_U, \eta_U})(U) d\mu(U) \in \npsi$. Then, for almost every $U$ in the support of $(\xi_U)_U$, we have $\eta_U \in \Dom(E_U)$.
\end{enumerate}
\end{Theorem}
\begin{proof} We only give a proof of the first statement. Consider the sesquilinear form
\[
q(\eta, \eta') = \varphi\left( \int (\iota \otimes \omega_{\xi_U, \eta_U})(U^\ast) d\mu(U) ^\ast \int (\iota \otimes \omega_{\xi_U, \eta_U'})(U^\ast) d\mu(U) \right), 
\]
with
\[
q(\eta) = q(\eta,\eta), \quad \Dom(q) = \left\{ \eta = \int^\oplus \eta_U d\mu(U) \mid \int (\iota \otimes \omega_{\xi_U, \eta_U})(U^\ast) d\mu(U) \in \nphi\right\}.
\]
$q$ is a closed form on $H$. Indeed, assume that $\eta_n\in \Dom(q)$ converges in norm to $\eta \in H$ and that $q(\eta_n - \eta_m) \rightarrow 0$. Then  $\int (\iota \otimes \omega_{\xi_U, \eta_{n,U}})(U^\ast) d\mu(U)$ converges to $\int (\iota \otimes \omega_{\xi_U, \eta_{U}})(U^\ast) d\mu(U)$ $\sigma$-weakly. By assumption $\Lambda(\int (\iota \otimes \omega_{\xi_U, \eta_{n,U}})(U^\ast) d\mu(U))$ is a Cauchy sequence in norm. The $\sigma$-weak-weak closedness of $\Lambda$ implies that $\int (\iota \otimes \omega_{\xi_U, \eta_{U}})(U^\ast) d\mu(U) \in \Dom(\Lambda) = \nphi$, so $\eta \in \Dom(q)$ and $\Lambda(\int (\iota \otimes \omega_{\xi_U, \eta_{n,U}})(U^\ast) d\mu(U))$ converges to $\Lambda(\int (\iota \otimes \omega_{\xi_U, \eta_{U}})(U^\ast) d\mu(U))$ weakly. Since we know that $\Lambda(\int (\iota \otimes \omega_{\xi_U, \eta_{n,U}})(U^\ast) d\mu(U))$ is a actually a Cauchy sequence in the norm topology it is norm convergent to $\Lambda(\int (\iota \otimes \omega_{\xi_U, \eta_{U}})(U^\ast) d\mu(U))$. This proves that $q(\eta - \eta_n) \rightarrow 0$.

Since $(\xi_U)_U$ is a square integrable, essentially bounded
field of vectors, $\int^\oplus \xi_U  \otimes \overline{ \eta_U}
d\mu(U) \in B_2(H)$. By Lemma \ref{LemCompI}, $\Dom(D) \subseteq
\Dom(q)$, so that $q$ is densely defined. $q$ is symmetric and
positive by its definition. By \cite[Theorem VI.2.23]{Kat}, there
is a unique positive, self-adjoint, possibly unbounded operator $A$ on $H$
such that $q(\eta, \eta') = \langle A \eta, A \eta' \rangle$ and
$\Dom(A) = \Dom(q)$. By Theorem \ref{ThmOrthogonalitRel} we see
that for $\eta, \eta'\in \Dom(D)$ we have $\int \langle D_U
\eta_U, D_U \eta_U' \rangle \Vert \xi_U \Vert^2 d\mu(U) = \langle
A \eta, A \eta' \rangle$. Since both $A$ and $\int^\oplus \Vert
\xi_U \Vert d\mu(U)$ are positive, self-adjoint operators this
yields $A = \int^\oplus \Vert \xi_U \Vert D_U d\mu(U)$. In
particular $\eta_U \in \Dom(D_U)$ for almost every $U\in {\rm
supp}\left((\xi_U)_U \right) = \overline{\left\{U \in \IC(M) \mid
\Vert \xi_U \Vert \not = 0  \right\}}$. \end{proof}

\section{Modular properties of matrix coefficients}\label{SectModular}

In this section we work towards expressions for the modular automorphism group of the left and right Haar weight in terms of matrix elements of corepresentations, culminating in Theorem \ref{ThmModularExpression}. The matrix coefficients of corepresentations are preserved under the modular automorphism group. The idea of proving this formula is to describe the polar decomposition of the conjugation operator $\Gamma(x) \mapsto \Lambda(x^\ast), x \in \npsi \cap \nphi^\ast$ explicitly in terms of corepresentations. Then, for a unimodular quantum group, where $\Gamma = \Lambda$, the modular automorphism group is implemented by the absolute value of this operator. 

Recall that in this section we use the notational conventions of Notation \ref{NtNotationDrieEnVier}.

\vspace{0.3cm}

At this point we recall the relevant results from the theory of normal, semi-finite, faithful (n.s.f.) weights and their modular automorphism groups. This is contained in \cite[Chapters VI, VII, VIII]{TakII}. We emphasize that the notation sometimes differs from \cite{TakII}.

Consider the following two operators \cite[Section VIII.3]{TakII}
\begin{equation}\label{EqnConjugationI}
\begin{split}
\conj_{\psi,0}: \hpsi \rightarrow \hpsi: \Gamma(x) &\mapsto \Gamma(x^\ast), \quad x \in \npsi \cap \npsi^\ast, \\
\conj_0: \hpsi \rightarrow \hphi: \Gamma(x) &\mapsto \Lambda(x^\ast), \quad x \in \npsi \cap \nphi^\ast.
\end{split}
\end{equation}
\noindent Both operators are densely defined and preclosed. We denote their closures by $\conj_\psi$ and $\conj$, respectively. $\conj_\psi$ and $\conj$ correspond to $S_\psi$ and $S_{\varphi, \psi}$ in \cite[Section VIII.3, (13)]{TakII}. We denote their polar decompositions by $\conj_\psi = J_\psi \NHalf_\psi, \conj = J \NHalf$. By construction, $J_\psi$ and $\nabla_\psi$ are the modular conjugation and the modular operator appearing in the Tomita-Takesaki theorem. In particular, $\nabla_\psi$ implements the modular automorphism group $\sigma_t^\psi$, i.e.
\begin{equation}\label{EqnModularAutomorphism}
\sigma_t^\psi(\pipsi(x)) = \nabla_\psi^{it} \pipsi(x) \nabla_\psi^{-it}, \qquad x \in M,\: t \in \mathbb{R}.
\end{equation}
\noindent Furthermore, $\nabla_\psi^{it}, t \in \mathbb{R}$, is a homomorphism of the left Hilbert algebra $\Gamma(\npsi \cap \npsi^\ast)$, i.e.
\begin{equation}\label{EqnHilbertProduct}
\nabla_\psi^{it} \pipsi(x) \nabla_\psi^{-it} \Gamma(y) = \pi_l(\nabla_\psi^{it} \Gamma(x)) \Gamma(y),
 \end{equation}
\noindent where $\pi_l(a)b = ab$ for $a, b \in \Gamma(\npsi \cap \npsi^\ast)$. By \cite[Section VIII.3, (11) and (29)]{TakII} the modular automorphism group $\sigma_t^\varphi$ is implemented by $\nabla$, i.e.
\begin{equation}\label{EqnImplementation}
\sigma_t^\varphi(x) = \nabla^{it} x \nabla^{-it}, \qquad x \in M,\: t \in \mathbb{R}.
\end{equation}
\noindent We emphasize that in general $\nabla^{it}, t \in \mathbb{R}$, fails to be a Hilbert algebra homomorphism of the left Hilbert algebra $\Gamma(\npsi \cap \npsi^\ast)$. In case $(M, \Delta)$ is unimodular, we find that $\nabla = \nabla_\psi$ and $\nabla^{it}, t \in \mathbb{R}$, satisfies the relation (\ref{EqnHilbertProduct}). This fact will eventually lead to Theorem \ref{ThmModularExpression}. However, we present the theory more general and do not suppose that $(M, \Delta)$ is unimodular until this theorem.

It turns out that the polar decomposition of $\conj$ can be expressed in terms of corepresentations by means of the Plancherel theorems. The polar decomposition of $\Q_L \circ \conj \circ \Q_R^{-1}$ and the morphisms $\Q_L$ and $\Q_R$ give the polar decomposition of $\conj$. Eventually this yields Theorems \ref{ThmDecompositionI} and \ref{ThmDecompositionII}.

\begin{Remark}\label{RmkNotationDom}  For $\xi = \int ^\oplus\xi_U
d\mu(U)\in H$ and $\eta = \int^\oplus \eta_U d\mu(U)\in H$, and $A
= \int^\oplus A_U d\mu(U)$, $B=\int^\oplus B_U d\mu(U)$ decomposable
operators on $H$, we will use $(\xi, \overline{\eta})\in \Domt(A,\overline{B})$ to mean
 $\xi \in \Dom(A), \overline{\eta} \in
\Dom(\overline{B})$, $(\xi_U \otimes \overline{\eta_U})_U$ is
square integrable and $\int^\oplus\xi_U\otimes \overline{\eta_U}
d\mu(U) \in \Dom(\int ^\oplus(A_U \otimes
\overline{B_U})d\mu(U))$. For closed opeators $A$ and $B$ the set of $\int^\oplus \xi_U \otimes \overline{\eta_U} d\mu(U)$ with  $(\xi_U, \overline{\eta_U})
\in \Domt(A,\overline{B})$ is a core for $\int^\oplus (A_U \otimes \overline{B_U}) d\mu(U)$ by Lemma \ref{LemTensorCore}. In particular this set is dense in $\int^\oplus H_U \otimes
\overline{H_U} d\mu(U)$. 
 
Let $\Sigma$ be the anti-linear flip
$
\Sigma: \int ^\oplus H_U \otimes \overline{H_U} d\mu(U) \rightarrow \int ^\oplus H_U \otimes \overline{H_U}  d\mu(U):\:\:
\int ^\oplus \xi_U \otimes \overline{\eta_U} d\mu(U) \mapsto
\int ^\oplus \eta_U \otimes \overline{\xi_U} d\mu(U).
$
\noindent $\Sigma$ is an anti-linear isometry of
$\int ^\oplus H_U \otimes \overline{H_U} d\mu(U)$.
\end{Remark}

\begin{Lemma}\label{LemCompIII}
For $\eta = \int ^\oplus_{\ICM} \eta_U d\mu(U), \xi =
\int^\oplus_{\ICM} \xi_U d\mu(U) \in H$, with $(\eta, \overline{\xi}) \in \Domt(E^{-1}, \overline{D}) )$, we have $\Q_R^{-1}\left(
\int^\oplus_{\ICM} \xi_U  \otimes \overline{ \eta_U}   d\mu(U)\right) \in \Dom(S)$
and:
\begin{equation}\label{EqnConjPlancherel}
\Q_L \circ \conj \circ \Q_R^{-1}\left( \int^\oplus_{\ICM} \xi_U  \otimes \overline{ \eta_U}   d\mu(U)\right) = \left( \int^\oplus_{\ICM} E_U^{-1}\eta_U  \otimes \overline{ D_U \xi_U}   d\mu(U)\right).
\end{equation}
\end{Lemma}

\begin{proof} By Lemma \ref{LemCompII}:
\begin{equation}\label{EqnLemCompIII}
\Q_R^{-1}\left( \int^\oplus \xi_U  \otimes \overline{ \eta_U}
d\mu(U)\right) = \Gamma\left( \int  (\iota \otimes \omega_{\xi_U,
E_U^{-1} \eta_U})(U) d\mu(U)\right).
\end{equation}
By Lemmas \ref{LemCompI} and \ref{LemCompII} we obtain
$$\int  (\iota \otimes
\omega_{\xi_U, E_U^{-1} \eta_U})(U) d\mu(U) = \left(\int  (\iota \otimes
\omega_{E_U^{-1} \eta_U, \xi_U})(U^\ast) d\mu(U)\right)^\ast \in \npsi
\cap\nphi^\ast.$$
Hence, by (\ref{EqnConjugationI}), (\ref{EqnLemCompIII}) and Lemma \ref{LemCompI}
\[
\begin{split}
\Q_L \circ \conj \circ \Q_R^{-1}\left( \int^\oplus \xi_U  \otimes \overline{ \eta_U}  d\mu(U)\right) & =  \Q_L \left(\Lambda\left( \int  (\iota \otimes \omega_{\xi_U,E_U^{-1} \eta_U})(U) d\mu(U)^\ast\right) \right)\\
& =  \left( \int^\oplus E_U^{-1}\eta_U  \otimes \overline{ D_U \xi_U}
d\mu(U)\right),  
\end{split} 
\]
from which the lemma follows.
\end{proof}

We are now able to give the polar decomposition of  $\Q_L \circ \conj \circ \Q_R^{-1}$.

\begin{Theorem}\label{ThmPolarDecI}
Consider $\conj_\Q := \Q_L \circ \conj \circ \Q_R^{-1}$  as an
operator on $\int^{\oplus}_{\ICM}  H_U \otimes \overline{H_U} d\mu(U)$.
Then the polar decomposition of $\conj_\Q$ is given by the
self-adjoint, strictly positive operator $\int^{\oplus}_{\ICM}  D_U
\otimes \overline{E_U^{-1}} d\mu(U)$ and the anti-linear isometry
$\Sigma$.
\end{Theorem}
\begin{proof} Throughout this proof, let $\eta = \int^\oplus
\eta_U d\mu(U) \in H$, $\xi = \int^\oplus \xi_U d\mu(U) \in H$,
$\eta' = \int^\oplus \eta_U' d\mu(U) \in H$ and $\xi' =
\int^\oplus \xi_U' d\mu(U) \in H$ be such that $(\eta_U \otimes
\overline{\xi_U})_U$ and $(\eta_U' \otimes \overline{\xi_U'})_U$
are square integrable.

Assume $(\eta, \overline{\xi}) \in \Domt(D,
\overline{E^{-1}})$, $(\xi', \overline{\eta'}) \in \Domt(D,
\overline{E^{-1}})$, so that by (\ref{EqnConjPlancherel}),
\[
\begin{split}
&\langle \int^\oplus  (\xi_U \otimes \overline{\eta_U}) d\mu(U), \conj_\Q \int^\oplus  (\xi_U' \otimes \overline{\eta_U'}) d\mu(U) \rangle  = \\ 
& \langle
\int^\oplus  (\xi_U \otimes \overline{\eta_U}) d\mu(U), \int^\oplus  (E_U^{-1}\eta_U' \otimes \overline{D_U \xi_U'}) d\mu(U) \rangle  = \\ &
 \int  \langle \xi_U, E_U^{-1} \eta_U' \rangle \langle D_U \xi_U', \eta_U \rangle d\mu(U)  = \\
& \int  \langle E_U^{-1} \xi_U,  \eta_U' \rangle \langle  \xi_U', D_U\eta_U \rangle  d\mu(U)   = \\ &  
 \langle \int^\oplus  (\xi_U' \otimes \overline{\eta_U'})
d\mu(U),\int^\oplus  (D_U \eta_U \otimes \overline{E_U^{-1}\xi_U})
d\mu(U) \rangle.
\end{split}
\]
\noindent So $\conj_\Q ^\ast \left( \int^\oplus  (\xi_U \otimes
\overline{\eta_U}) d\mu(U) \right) = \int^\oplus  (D_U \eta_U
\otimes \overline{E_U^{-1}\xi_U}) d\mu(U)$.

Assuming $(\xi, \overline{\eta}) \in \Domt(D^2, \overline{E^{-2}})$, it follows
$$\conj_\Q^\ast \conj_\Q\left( \int^\oplus  (\xi_U \otimes \overline{\eta_U}) d\mu(U) \right) = \int^\oplus  (D_U^2\xi_U \otimes \overline{E_U^{-2}\eta_U})
d\mu(U).$$
\noindent $\int ^\oplus D_U^2 \otimes \overline{E_U^{-2}} d\mu(U)$ is a
positive, self-adjoint operator for which the set
$$C := \textrm{span}_\mathbb{C} \left\{ \int ^\oplus (\xi_U
\otimes \overline{\eta_U}) d\mu(U)  \mid (\xi, \eta) \in
\Domt(D^2, E^{-2}) \right\},$$
\noindent forms a core by Lemma \ref{LemTensorCore}. Since $\conj_\Q^\ast \conj_\Q$ is self-adjoint and agrees with the self-adjoint operator $\int
^\oplus D_U^2 \otimes \overline{E_U^{-2}} d\mu(U)$ on $C$ we find
$\conj_\Q^\ast \conj_\Q = \int ^\oplus D_U^2 \otimes
\overline{E_U^{-2}} d\mu(U)$.

Assuming that $(\xi, \overline{\eta}) \in \Domt(D,
\overline{E^{-1}})$,
\[
\begin{split}
&\Sigma \circ \left(\int ^\oplus D_U \otimes \overline{E_U^{-1}}
d\mu(U)\right) \left( \int^\oplus  (\xi_U \otimes \overline{\eta_U}) d\mu(U) \right) =\\  &
\Sigma \left( \int^\oplus  (D_U \xi_U \otimes \overline{E_U^{-1} \eta_U}) d\mu(U) \right)   =
\int^\oplus  (E_U^{-1} \eta_U \otimes \overline{D_U \xi_U})
d\mu(U)  ,
\end{split}
\]
\noindent so that $\conj_\Q$ and $\Sigma \circ \left(\int^\oplus
D_U \otimes \overline{E_U^{-1}} d\mu(U)\right)$ agree on a core,
cf. Remark \ref{RmkNotationDom}. \end{proof}

Finally we translate everything back to the level of the
GNS-representations $\hphi$ and $\hpsi$.

%and suppose that $C_p := \left\{ \xi_p \mid \textrm{there exists a vector} \int^{\oplus} \eta_r d\mu(r) \in D \textrm{ s.t. } \eta_p = \xi_p \right\}$ is
%dense in $D_p$ for every $p$.

\begin{Proposition}\label{PropNabla}
Let
\[
\begin{split}
\dnablanot = & {\rm span}_\mathbb{C} \{ \int_{\ICM} 
\el{U}{\xi_U}{\eta_U} d\mu(U) \mid {\rm where }   \\ &   \qquad \qquad \eta \in \Dom(E) \cap
\Dom(E^{-1}), (\xi, E \eta) \in \Domt(D,E^{-1}) \},
\end{split}
\]
and define $\NHalf_0 :\: \Gamma(\dnablanot) \rightarrow \hpsi$ by
$$ \Gamma(\int_{\ICM}  \el{U}{\xi_U}{\eta_U} d\mu(U)) \mapsto  \Gamma(\int_{\ICM}  \el{U}{D_U \xi_U}{E_U^{-1}\eta_U} d\mu(U)).$$
Then $\NHalf_0$ is a densely defined, preclosed
operator and its closure $\NHalf$, is a self-adjoint,
strictly positive operator satisfying $\Q_R \circ \NHalf \circ
\Q_R^{-1} = \int^{\oplus}_{\ICM}   D_U \otimes \overline{E_U^{-1}} d\mu(U)
$.
\end{Proposition}
\begin{proof} Let $C := \textrm{span}_\mathbb{C}\left\{ \int^\oplus \xi_U
\otimes \overline{\eta_U} d\mu(U) \mid (\xi, \eta)\in \Domt(D_U^2,
E_U^{-2}) \right\}$. Then $C$ is a core for $\int^{\oplus} D_U
\otimes \overline{E_U^{-1}} d\mu(U)$. Indeed, $C$ is a core for
$\int^{\oplus}  D_U^2 \otimes \overline{E_U^{-2}} d\mu(U)$ by
Lemma \ref{LemTensorCore}, and hence this is a core for
$\int^{\oplus}  D_U \otimes \overline{E_U^{-1}} d\mu(U)$.

Now, let $\eta = \int^\oplus \eta_U d\mu(U)\in H$ and $\xi = \int^\oplus \xi_U d\mu(U)\in H$ be such that
$$\eta \in \Dom(E) \cap \Dom(E^{-1}), \qquad (\xi, E \eta) \in \Domt(D,E^{-1}).$$
\noindent So $\eta \in \Dom(E)$ and $(\xi_U \otimes E_U \eta_U)_U$ is square integrable, so that $\int \el{U}{\xi_U}{\eta_U} d\mu(U) \in \npsi$ by Lemma
\ref{LemCompII}. Similarly, since $E^{-1}\eta \in \Dom(E)$ and  $(D_U \xi_U \otimes \eta_U)_U$ is square integrable, $\int \el{U}{D_U
\xi_U}{E_U^{-1}\eta_U} d\mu(U)) \in \npsi$. Furthermore, we have the following inclusions:
$$C \subseteq \Q_R( \Gamma(\dnablanot)) \subseteq
\Dom(\int^\oplus D_U \otimes \overline{E_U^{-1}}d\mu(U) )$$
and for
$x \in \Gamma(\dnablanot)$ we have,
$$\NHalf_0 (x) = \Q_R^{-1}\left(
\int^{\oplus}  D_U \otimes \overline{E_U^{-1}} d\mu(U)\right)
\Q_R(x) .$$
Since $\Q_R$ is an isometric isomorphism, the claims
 follow from the fact that $\int^{\oplus}  D_U \otimes
\overline{E_U^{-1}} d\mu(U)$ is a self-adjoint, strictly positive
operator for which $C$ is a core.\end{proof}

\begin{Proposition}\label{PropJee}
Let $\djnot$ be the linear space
\[
\begin{split}
& {\rm span}_\mathbb{C} \{ \int_{\ICM}^\oplus
\el{U}{\xi_U}{\eta_U} d\mu(U) \mid {\rm where } \\ & \qquad \qquad \xi \in \Dom(D^{-1}), \eta \in
\Dom(E),  (\xi_U  \otimes \overline{ E_U \eta_U})_U \:\: {\rm is\:\:square\:\: integrable}    \},
\end{split}
\]
and define
$J_0 : \Gamma(\djnot) \rightarrow \hphi:$
\[ \Gamma(\int_{\ICM}  \el{U}{\xi_U}{\eta_U} d\mu(U)) \mapsto  \Lambda(\int_{\ICM} \el{U}{D_U^{-1} \xi_U}{E_U\eta_U} d\mu(U)^\ast).
\]
Then $J_0$ is a densely defined anti-linear isometry, and
its closure, denoted by $J$, is a surjective anti-linear isometry satisfying
$\Q_L \circ J \circ \Q_R^{-1} = \Sigma$.

\end{Proposition}
\begin{proof} Let $C := \textrm{span}_\mathbb{C}\left\{ \int
\xi_U \otimes \overline{\eta_U} d\mu(U) \mid (\xi, \eta) \in
\Domt(D^{-1}, \overline{E}) \right\}$. $C$ is dense in $\int^\oplus H_U
\otimes \overline{H_U} d\mu(U)$, c.f. Remark \ref{RmkNotationDom}.

For $\eta = \int^\oplus \eta_U d\mu(U)\in H$ and $\xi = \int^\oplus
\xi_U d\mu(U)\in H$ so that $\xi \in \Dom(D^{-1})$, $\eta \in
\Dom(E)$ and $(\xi_U  \otimes \overline{ E_U \eta_U})_U$ is square integrable, we find
 $\int  \el{U}{\xi_U}{\eta_U} d\mu(U) \in \npsi$ and $\int
\el{U}{D_U^{-1} \xi_U}{E_U\eta_U} d\mu(U)^\ast \in \nphi$ by Lemmas \ref{LemCompI} and \ref{LemCompII}. So $C \subseteq \Q_R(\Gamma(\djnot))$, and for $x \in
\Gamma(\djnot)$, $J_0(x) = \Q_L^{-1} \circ \Sigma \circ \Q_R(x).$
Then, since
$\Q_L$ and $\Q_R$ are isomorphisms, the claim follows from
 $\Sigma$ being a surjective anti-linear isometry.\end{proof}

Note that the previous proposition is an analogy of the classical situation. Suppose that $G$ is a locally compact group for which the classical Plancherel theorem \cite[Theorem 18.8.1]{DixC} holds. The anti-linear operator $f \mapsto f^\ast$  acting on $L^2(G)$ is transformed into the anti-linear flip acting on $\int^\oplus K(\zeta) \otimes \overline{K}(\zeta) d\mu(\zeta)$ by the Plancherel transform. Here $f^\ast(x) = \overline{f(x^{-1})} \delta_G(x^{-1})$ and $\delta_G$ is the modular function on $G$.

From Theorem \ref{ThmPolarDecI} and
Propositions \ref{PropNabla} and \ref{PropJee} we obtain the following result.

\begin{Theorem}\label{ThmDecompositionI}
The polar decomposition of $\conj$ is given by $S = J \NHalf$.
\end{Theorem}

The roles of $\varphi$ and $\psi$ can be interchanged. Consider the
operator:
\begin{equation}\label{EqnConjugationIII}
\conj_0': \hphi \rightarrow \hpsi: \Lambda(x) \mapsto \Gamma(x^\ast), \quad x \in \nphi \cap \npsi^\ast.
\end{equation}
\noindent This operator is densely defined and preclosed. We denote its closure by $\conj'$. The polar decomposition of $\conj'$ can be expressed in terms of corepresentations in a
similar way.

\begin{Theorem}\label{ThmDecompositionII}
Consider $\conj': \hphi \rightarrow \hpsi$. Let
$\djnot'$ be the linear space
\[
\begin{split}
& {\rm span}_\mathbb{C} \{ \int^\oplus_{\ICM}
\el{U}{\xi_U}{\eta_U}^\ast d\mu(U) \mid {\rm where }\\ & \qquad \qquad  \xi \in \Dom(D), \eta \in
\Dom(E^{-1}),  (D_U \xi_U \otimes \overline{ \eta_U} )_U \:\: {\rm is\:\: sq. \:\: int.}  \},
\end{split}
\] 
and define $J_0' : \Lambda(\djnot') \rightarrow \hpsi$: $$\Lambda(\int_{\ICM}  \el{U}{\xi_U}{\eta_U}^\ast d\mu(U)) \mapsto  \Gamma(\int_{\ICM} \el{U}{D_U \xi_U}{E_U^{-1}\eta_U} d\mu(U)).$$
Then $J_0'$ is densely defined and isometric, and
its closure, denoted by $J'$, is a surjective anti-linear isometry. Let
\[
\begin{split}
& \dnablanot' = \textrm{span}_\mathbb{C} \{ \int_{\ICM}
\el{U}{\xi_U}{\eta_U} d\mu(U) \mid {\rm where } \\ & \qquad \qquad \xi \in \Dom(D) \cap
\Dom(D^{-1}), (D \xi, \overline{\eta}) \in \Domt(D^{-1},\overline{E}) \},
\end{split}
\]
 and
define $\NHalfAcc_0 : \Lambda(\dnablanot) \rightarrow \hpsi$: $$\Lambda(\int_{\ICM} \el{U}{\xi_U}{\eta_U} d\mu(U)^\ast) \mapsto  \Lambda(\int_{\ICM} \el{U}{D_U^{-1} \xi_U}{E_U\eta_U} d\mu(U)^\ast).$$
Then $\NHalfAcc_0$ is a densely defined, preclosed
operator and its closure, denoted by $\NHalfAcc$, is a
self-adjoint, strictly positive operator.

Moreover, the polar decomposition of $\conj'$ is given by $\conj'
= J' \NHalfAcc$.
\end{Theorem}

We now assume that $(M, \Delta)$ is unimodular, so that $S = S'= S_\psi$ and Theorem \ref{ThmDecompositionI} give an explicit expression for the modular operator and modular conjugation. This leads to the following expression for the modular automorphism group. In this case we write $\sigma_t$ for $\sigma_t^\varphi = \sigma_t^\psi$.

\begin{Theorem}\label{ThmModularExpression}
Suppose that $(M, \Delta)$ is unimodular. Let $(\xi_U)_U, (\eta_U)_U$ be square integrable vector fields. The modular automorphism group $\sigma_t$ of the Haar weight $\psi$ can be expressed as:

\begin{equation}\label{EqnModularExpression}
\sigma_t \left( \int_{\ICM} \el{U}{\xi_U}{\eta_U}d\mu(U) \right) = \int_{\ICM} \el{U}{D_U^{2it}\xi_U}{E_U^{2it}\eta_U}d\mu(U).
\end{equation}
\end{Theorem}

\begin{proof} For $\eta = \int^\oplus \eta_U d\mu(U) \in H$, $\xi =
\int^\oplus \xi_U d\mu(U) \in H$, such that $( \xi_U \otimes \overline{\eta_U})_U$ is a square integrable field of vectors and $\eta \in \Dom(E)$, we find
\begin{equation}\label{EqnModularII}
\nabla^{it} \Gamma\left(\int \el{U}{\xi_U}{\eta_U} d\mu(U)\right) = \Gamma\left(\int \el{U}{D_U^{2it}\xi_U}{E_U^{2it}\eta_U} d\mu(U)\right).
\end{equation}
Indeed, $\left( \int^\oplus (D_U\otimes \overline{E_U^{-1}})d\mu(U)\right)^{2it} (\xi \otimes \overline{\eta}) = \int^\oplus (D_U^{2it}\xi_U\otimes
\overline{E_U^{2it}\eta_U})d\mu(U)$ by \cite[Theorem 1.10]{Lan}, so (\ref{EqnModularII}) follows from Lemma \ref{LemCompII} and Proposition \ref{PropNabla}. Since $\sigma_t(\pipsi(x)) = \nabla^{it} \pipsi(x) \nabla^{-it}, x \in M$, (\ref{EqnHilbertProduct}) implies
\begin{equation}\label{EqnModExpression}
\sigma_t \left(\pipsi \left( \int \el{U}{\xi_U}{\eta_U} d\mu(U) \right)\right)  =  \pipsi \left( \int \el{U}{D_U^{2it}\xi_U}{E_U^{2it}\eta_U} d\mu(U) \right),
\end{equation}
\noindent so the theorem follows from the identification of $M$ with $\pipsi(M)$, in this case.

Now let $\eta = \int^\oplus \eta_U d\mu(U) \in H$ and $\xi =
\int^\oplus \xi_U d\mu(U) \in H$ be arbitrary. We take sequences of square integrable vector fields $\xi_n = \int^\oplus \xi_{U,n} d\mu(U)$,
$\eta_n = \int^\oplus \eta_{U,n} d\mu(U)$ such that $(\xi_{U,n} \otimes \overline{\eta_{U,n}})_U$ is a square integrable field of vectors, $\eta_n \in \Dom(E)$ and such that $\xi_n$ converges to $\xi$ and $\eta_n$  converges to $\eta$. Then $\int \el{U}{\xi_{U,n}}{\eta_{U,n}} d\mu(U)$ is $\sigma$-weakly convergent to $\int \el{U}{\xi_{U}}{\eta_{U}} d\mu(U)$ and hence
\[
\begin{split}
& \sigma_t\left(\int \el{U}{\xi_{U}}{\eta_{U}} d\mu(U) \right) = \lim_{n\rightarrow\infty} \sigma_t\left(\int \el{U}{\xi_{U,n}}{\eta_{U,n}} d\mu(U)
\right) = \\
 & \lim_{n\rightarrow\infty} \left(\int \el{U}{D_U^{2it}\xi_{U,n}}{E_U^{2it}\eta_{U,n}} d\mu(U) \right) 
 =  \int \el{U}{D_U^{2it}\xi_{U}}{E_U^{2it}\eta_{U}} d\mu(U) ,
\end{split}
\]
which yields (\ref{EqnModularExpression}).
\end{proof}

We used (\ref{EqnHilbertProduct}) to obtain (\ref{EqnModExpression}). The unimodularity assumption is essential for Theorem \ref{ThmModularExpression}.

\begin{Corollary}
 Let $(M, \Delta)$ be unimodular. Let $\eta = \int^\oplus
\eta_U d\mu(U) \in H$, $\xi = \int^\oplus \xi_U d\mu(U) \in H$, $r \in \mathbb{R}$ be such that $\eta \in \Dom(E^{2r})$ and $\xi \in \Dom(D^{2r})$, then:
$$\int_{\ICM} \el{U}{\xi_U}{\eta_U}d\mu(U) \in \Dom(\sigma_z),$$
\noindent for all $z$ in the strip $S(r) := \left\{ z \in \mathbb{C} \mid 0 \leq \textrm{Im}(z) \leq r, \textrm{ or } r \leq \textrm{Im}(z) \leq 0 \right\}$. In particular, if $\eta$ is analytic for $E$ and if $\xi$ is analytic for $D$, then $\int_{\ICM} \el{U}{\xi_U}{\eta_U}d\mu(U)$ is analytic for the one-parameter group $\sigma_t$.
\end{Corollary}
\begin{proof}  For $\alpha \in M_\ast$, define
\[
\begin{split}
&F_\alpha (z) = \alpha \left(\int \el{U}{D_U^{2iz} \xi_U}{E_U^{2i\overline{z}}\eta_U}d\mu(U) \right) = \\
&\langle \int^\oplus \!\!\!\! (\alpha \otimes \iota)(U) d\mu(U) (\int^\oplus  \!\!\!\! D_U d\mu(U) )^{2iz} \int^\oplus  \!\!\!\! \xi_U d\mu(U), (\int^\oplus \!\!\!\! E_U d\mu(U) ) ^{2i\overline{z}} \int^\oplus  \!\!\!\!\eta_U d\mu(U)\rangle.
\end{split}
\]
\noindent Here the last equality follows from \cite[Theorem 1.10]{Lan}. By \cite[Lemma VI.2.3]{TakII}, $F_\alpha (z)$ is an analytic continuation of $\alpha\left(\sigma_t^\varphi\left(\int \el{U}{\xi_U}{\eta_U}d\mu(U)\right)\right)$ to the strip $S(r)$ such that $F_\alpha(z)$ is bounded by a constant $C \Vert \alpha \Vert$ where $C$ is independent of $\alpha$. Moreover, $F_\alpha(z)$ is continuous on $S(r)$ and analytic on the interior $S(r)^\circ$. Therefore $F(z) = \int \el{U}{D_U^{2iz} \xi_U}{E_U^{2i\overline{z}}\eta_U}d\mu(U) $ is a continuation of $\sigma_t \left(\int \el{U}{\xi_U}{\eta_U}d\mu(U)\right)$  to the strip $S(r)$ such that $F(z)$ is bounded and $\sigma$-weakly continuous on $S(r)$ and analytic on the interior $S(r)^\circ$ \cite[Result 1.2]{KusOneParam}.
 
\end{proof}
 
\section{Example}\label{SectExample}

Using the theory of square integrable corepresentions, Desmedt
\cite{Des} determined the operators $D_U$ and $E_U$ for the
corepresentions that appear as discrete mass points of the
Plancherel measure, see also Remark \ref{RmkSqIntCoreps}. In
particular, his theory applies to compact quantum groups, for which
every corepresentation is square integrable. As a non-compact
example, Desmedt was able to determine the operators $D_U$ for the
discrete series corepresentations of the quantum group analogue of
the normalizer of $SU(1,1)$ in $SL(2,\mathbb{C})$, which we
denote by $(M, \Delta)$ from now on, see \cite{KoeKus} and \cite{GrKoeKus}. Having the theory of Sections
\ref{SectPlancherel} and \ref{SectModular} at hand we determine
the operators $D_U$ and $E_U$ for the principal series
corepresentations of $(M, \Delta)$.

We refer to \cite{KoeKus} and \cite{GrKoeKus} for the relevant properties of $(M, \Delta)$ and use the
same notational conventions. In \cite[Theorem 5.7]{GrKoeKus} a
decomposition of the multiplicative unitary in terms of
irreducible corepresentations is given:
\begin{equation}\label{EqnWDecomposition}
W = \bigoplus_{p \in q^\mathbb{Z}} \left( \int_{[-1,1]}^\oplus
W_{p,x} dx \oplus \bigoplus_{x\in \sigma_d(\Omega_p)} W_{p,x}
\right).
\end{equation}
\noindent Here $\sigma_d(\Omega_p)$ is the discrete spectrum of the Casimir
operator \cite[Definition 4.5, Theorem 4.6]{GrKoeKus} restricted to the subspace given in
\cite[Theorem 5.7]{GrKoeKus}. $W_{p,x}$ is a corepresention that is a
direct sum of at most 4 irreducible corepresentations
\cite[Propositions 5.3 and 5.4]{GrKoeKus}. An orthonormal basis for the corepresentation
Hilbert space $\mathcal{L}_{p,x}$ of $W_{p,x}$ is given the
vectors $e^{\ep, \eta}_m (p,x), \ep, \eta \in \{ -, + \}, m \in
\mathbb{Z}$. The corepresentations $W_{p,x}, p \in q^\mathbb{Z}, x
\in \sigma(\Omega_p)$ are called the discrete series
corepresentations and the corepresentations $W_{p,x}, p \in
q^\mathbb{Z},  x \in [-1,1]$ are called the principal series
corepresentations. We denote $D_{p,x}$ and $E_{p,x}$ for
$D_{W_{p,x}}$ and $E_{W_{p,x}}$. The operators $D_{p,x}$ have been computed by
Desmedt \cite{Des} for the discrete series. Hence we focus on the
principal series. In Appendix \ref{AppendixB} we
verify that $(M, \Delta)$ satisfies the conditions of the
Plancherel theorem, so that the theory of Sections
\ref{SectPlancherel} and \ref{SectModular} applies.
Furthermore, $(M, \Delta)$ is unimodular \cite{KoeKus}. We denote the modular automorphism group of the Haar weight by $\sigma_t$.

By \cite[Lemmas 10.9]{GrKoeKus} the action of the matrix elements in the GNS-space can be calculated explicitly:
\[
\begin{split}
&(\iota \otimes \omega_{e^{\ep,\eta}_m, e^{\ep',\eta'}_{m'}})\left(W_{p,x} \right) f_{m_0, p_0, t_0}  = \\
& C(\eta\ep x;m',\ep',\eta';\ep\ep'\vert p_0\vert
p^{-1}q^{-m-m'},p_0,m-m') \delta_{sgn(p_0),\eta\eta'} f_{m_0 - m +
m',\ep\ep' \vert p_0\vert p^{-1} q^{-m-m'} , t_0}.
\end{split}
\]
Fix $p \in q^\mathbb{Z}$. Let $\ep, \eta, m, \ep', \eta', m'$ be
$\mu$-measurable functions of $x \in [-1,1]$, thus  $\ep = \ep(x),
\eta = \eta(x), \ldots$. Let $f, f'$ be $\mu$-square integrable complex functions on $[-1,1]$. Then $f(x) e^{\ep, \eta}_{m} = f(x) e^{\ep,
\eta}_{m}(p,x)$ and $ f'(x) e^{\ep', \eta'}_{m'} = f'(x) e^{\ep',
\eta'}_{m'}(p,x)$ are $\mu$-square integrable fields of vectors.
Since the modular automorphism group $\sigma_t$ is implemented by
$\gamma^\ast \gamma$ \cite[Section 4]{KoeKus}, Theorem
\ref{ThmModularExpression} yields
\begin{equation}\label{EqnDubbelModular}
\begin{split}
& \left( \int_{[-1,1]} (\iota \otimes \omega_{f(x)D_{p,x}^{2it} e^{\ep,\eta}_m, f'(x)E_{p,x}^{2it}e^{\ep',\eta'}_{m'}})\left(W_{p,x} \right) d\mu(x) \right) f_{m_0, p_0, t_0}  \\
 = & \sigma_t\left(  \int_{[-1,1]} (\iota \otimes \omega_{f(x)e^{\ep,\eta}_m,
f'(x)e^{\ep',\eta'}_{m'}})\left(W_{p,x} \right) d\mu(x) \right) f_{m_0,
p_0, t_0 } \\
 = &\vert \gamma \vert^{2it} \left( \int_{[-1,1]} (\iota \otimes \omega_{f(x)e^{\ep,\eta}_m, f'(x)e^{\ep',\eta'}_{m'}})\left(W_{p,x} \right) d\mu(x) \right) \vert \gamma \vert^{-2it} f_{m_0, p_0, t_0}  \\
= & \left(\frac{p_0^2}{p_0^2p^{-2}q^{-2m-2m'}}\right)^{it} \int_{[-1,1]}
 \!\!\!\! \!\!\!\! f(x)\overline{f'(x)}C(\eta\ep x;m',\ep',\eta';\ep\ep'\vert p_0\vert p^{-1}q^{-m-m'},p_0,m-m') \\
&\qquad\qquad\times \delta_{sgn(p_0),\eta\eta'} f_{m_0 - m + m',\ep\ep' \vert p_0\vert p^{-1} q^{-m-m'} , t_0} d\mu(x) \\
= & (p^{2}q^{2m+2m'})^{it} \int_{[-1,1]} (\iota \otimes
\omega_{f(x)e^{\ep,\eta}_m, f'(x)e^{\ep',\eta'}_{m'}})\left(W_{p,x} \right)
d\mu(x)  f_{m_0, p_0, t_0}.
\end{split}
\end{equation}
Define $A$ and $B$ as the unbounded self-adjoint operators on
$\int^\oplus_{[-1,1]} \mathcal{L}_{p,x} d\mu(x)$ determined by $A
= \int^\oplus_{[-1,1]} A_{p,x} d\mu(x)$, $A_{p,x} e_m^{\ep,
\eta}(p,x) = p^2q^{2m} e_m^{\ep, \eta}(p,x)$. $B =
\int^\oplus_{[-1,1]} B_{p,x} d\mu(x)$, $B_{p,x} e_m^{\ep,
\eta}(p,x) = q^{-2m} e_m^{\ep, \eta}(p,x)$.
So (\ref{EqnDubbelModular}) yields
\begin{equation}\label{EqnVoorbeeldGelijk}
\begin{split}
 & \int (\iota \otimes \omega_{f(x)D_{p,x}^{2it} e^{\ep,\eta}_m,
f'(x)E_{p,x}^{2it}e^{\ep',\eta'}_{m'}})\left(W_{p,x} \right) d\mu(x)
 \\ = & \int (\iota \otimes \omega_{f(x)A_{p,x}^{it}e^{\ep,\eta}_m,
f'(x)B_{p,x}^{it}e^{\ep',\eta'}_{m'}})\left(W_{p,x} \right) d\mu(x),
\end{split}
\end{equation}
where the integrals are taken over $[-1,1]$. For any two bounded operators $F = \int^\oplus_{[-1,1]} F_{p,x} d\mu(x)$, $G
= \int^\oplus_{[-1,1]} G_{p,x} d\mu(x)$ on $\int^\oplus_{[-1,1]}
\mathcal{L}_{p,x} d\mu(x)$, the map $\left( \int^\oplus_{[-1,1]}
\mathcal{L}_{p,x} d\mu(x)\right) \otimes
\overline{\left(\int^\oplus_{[-1,1]} \mathcal{L}_{p,x}
d\mu(x)\right)} \rightarrow M$ given by
\[
 v \otimes \overline{w} = \int^\oplus_{[-1,1]}
v_x d\mu(x) \otimes \overline{\int^\oplus_{[-1,1]} w_x d\mu(x)}
\mapsto \int_{[-1,1]} (\iota \otimes \omega_{F_{p,x} v_x, G_{p,x}
w_x})\left(W_{p,x} \right) d\mu(x)
\]
is norm-$\sigma$-weakly continuous since 
\[
\vert \int_{[-1,1]}
\alpha \otimes \omega_{v_x, w_x}(W_{p,x}) d\mu(x) \vert \leq \Vert
\alpha \Vert \Vert F \Vert \Vert G \Vert \Vert v \Vert \Vert w
\Vert, \alpha \in M_\ast.
\]
 Therefore, for $v =
\int^\oplus_{[-1,1]} v_x d\mu(x),  w = \int^\oplus_{[-1,1]} w_x
d\mu(x) \in \int^\oplus_{[-1,1]} \mathcal{L}_{p,x} d\mu(x)$, using \cite[II.1.6, Proposition 7]{Dix} and (\ref{EqnVoorbeeldGelijk}),
\begin{equation}\label{EqnEqualMatrixElements}
 \int_{[-1,1]} (\iota \otimes \omega_{D_{p,x}^{2it} v_x,
E_{p,x}^{2it}w_x })\left(W_{p,x} \right) d\mu(x)
  = \int_{[-1,1]} (\iota \otimes \omega_{A_{p,x}^{it}v_x,
B_{p,x}^{it} w_x})\left(W_{p,x} \right) d\mu(x).
\end{equation}
For $v = \int^\oplus_{[-1,1]} v_x d\mu(x), w =
\int^\oplus_{[-1,1]} w_x d\mu(x) \in \int^\oplus_{[-1,1]}
\mathcal{L}_{p,x} d\mu(x)$, with $(v_x)_x $ essentially bounded,
 $w \in \Dom\left(
\int^\oplus_{[-1,1]} E_{p,x} d\mu(x) \right)$, Theorem
\ref{ThmOrthogonalitRel} implies that $ \int_{[-1,1]} (\iota \otimes
\omega_{D_{p,x}^{2it} v_x, E_{p,x}^{2it}w_x })\left(W_{p,x}
\right) d\mu(x) \in \npsi. $ By (\ref{EqnEqualMatrixElements}) and Theorem
\ref{ThmDomain}, $B_{p,x}^{it} w_x \in \Dom(E_{p,x})$ almost
everywhere in the support of $(v_x)_x$. Theorem \ref{ThmOrthogonalitRel}
implies that for $v' = \int^\oplus_{[-1,1]} v_x' d\mu(x), w' =
\int^\oplus_{[-1,1]} w_x' d\mu(x)  \in \int^\oplus_{[-1,1]}
\mathcal{L}_{p,x} d\mu(x)$ with the extra assumptions $w' \in \Dom\left(
\int^\oplus_{[-1,1]} E_{p,x}^2 d\mu(x) \right)$ and $(v'_x \otimes
E_{p,x} w'_x)_x$ is square integrable,
\[
\begin{split}
&\int_{[-1,1]} \langle  B^{it}_{p,x} w_x, E_{p,x}^2 w_x' \rangle \langle v'_x,
A^{it}_{p,x} v_x \rangle d\mu(x) \\=&
\psi\left(\left( \int_{[-1,1]}   (\iota \otimes \omega_{A^{it}_{p,x} v_x, B^{it}_{p,x} w_x})(W_{p,x}) d\mu(x)\right) ^\ast\int_{[-1,1]}   (\iota \otimes \omega_{v_x', w_x'})(W_{p,x}) d\mu(x) \right) \\
=&  \psi\left(\left( \int_{[-1,1]}   (\iota \otimes \omega_{D_{p,x}^{2it} v_x,  E_{p,x}^{2it}w_x})(W_{p,x}) d\mu(x)\right) ^\ast\int_{[-1,1]}   (\iota \otimes \omega_{v_x', w_x'})(W_{p,x}) d\mu(x) \right)\\
=&  \int_{[-1,1]}  \langle  E_{p,x}^{2it} w_x, E_{p,x}^2  w_x' \rangle
\langle v_x', D_{p,x}^{2it} v_x \rangle d\mu(x).
\end{split}
\]
$E_{p,x}$ is strictly positive by the Plancherel theorem. The
elements $\int^\oplus_{[-1,1]} v'_x \otimes \overline{ E_{p,x}^2 w'_x} 
d\mu(x)$ are dense in $ \int^\oplus_{[-1,1]} \mathcal{L}_{p,x}
d\mu(x) \otimes \overline{ \int^\oplus_{[-1,1]} \mathcal{L}_{p,x}
d\mu(x)}$, so $ \int^\oplus_{[-1,1]} D_{p,x}^{2it} \otimes
\overline{E_{p,x}^{2it}} d\mu(x) = \int^\oplus_{[-1,1]}
A_{p,x}^{it} \otimes \overline{B_{p,x}^{it}} d\mu(x)$. By Stone's
theorem and \cite[Theorem 1.10]{Lan} $ \int^\oplus_{[-1,1]}
D_{p,x} \otimes \overline{E_{p,x}} d\mu(x) = \int^\oplus_{[-1,1]}
A_{p,x}^{\frac{1}{2}} \otimes \overline{B_{p,x}}^{\frac{1}{2}} d\mu(x)$. Hence we see that there is
a positive function $c(p,x)$, such that
\[
\begin{split}
D_{p,x} e^{\ep, \eta}_m &= p q^m c(p,x) e^{\ep, \eta}_m,\\
E_{p,x} e^{\ep, \eta}_m &= q^{-m} c(p,x) e^{\ep, \eta}_m.
\end{split}
\]
The function $c(p,x)$ depends on the choice of the Plancherel
measure $\mu$, see \cite[Theorem 3.4.1, part 6]{Des}. 
\begin{Remark}
Desmedt \cite[\S 3.5]{Des} obtains a similar result using summation formulas for basic hypergeometric series, a method different from the one presented here. Note that the present method also applies to discrete series corepresentations and avoids caclulations involving special functions.
\end{Remark}

\appendix

\section{Appendix}\label{AppendixA}

 For the theory of direct integrals of bounded operators we refer to \cite{Dix}. For the theory of direct integrals of unbounded closed operators we refer to \cite{Lan}, \cite{Nus} and \cite[Chapter 12]{Schm}.
\begin{Lemma}\label{LemTensorCore}
Let $(X, \mu)$ be a standard measure space. Let $(H_p)_p$ and
$(K_p)_p$ be measurable fields of Hilbert spaces. Let $(A_p)_p$
and $(B_p)_p$ be measurable fields of closed operators on
$(H_p)_p$ and $(K_p)_p$ respectively. Let $(e^n_p)_p, n\in
\mathbb{N}$ be a fundamental sequence for $(A_p)_p$ and let
$(f^n_p)_p, n\in \mathbb{N}$ be a fundamental sequence for
$(B_p)_p$. Set $A = \int^\oplus_X A_p d\mu(p)$, $B = \int^\oplus_X
B_p d\mu(p)$, $H = \int^\oplus_X H_p d\mu(p)$ and $K =
\int^\oplus_X K_p d\mu(p)$.
\begin{enumerate}
\item[(a)] $(A_p \otimes B_p)_p$ is a measurable field of closed operators.
\item[(b)] The countable set
$$R = \left\{ (e^n_p \otimes f^m_p)_p
\mid n,m \in \mathbb{N} \right\},$$ is a fundamental sequence for
$(A_p \otimes B_p)_p$.
\item[(c)] The set
$$T = {\rm span}_{\mathbb{C}}\left\{ \int^\oplus_X \!\!\!\! \xi_p \otimes \eta_p d\mu(p)  \mid \!\!\!\! \begin{array}{l} \xi = \int^\oplus_X \xi_p d\mu(p) \in \Dom(A), \\ \eta = \int^\oplus_X \eta_p d\mu(p) \in \Dom(B),\\ \int^\oplus_X(\xi_p\otimes \eta_p) d\mu(p) \in \Dom(\int ^\oplus(A_p
\otimes B_p)d\mu(p)) \end{array} \right\},$$
is a core for $\int^\oplus_X (A_p \otimes B_p) d\mu(p)$.
\end{enumerate}
\end{Lemma}
\begin{proof} We first prove (a) and (b).  By \cite[II.1.8,
Proposition 10]{Dix}, for $(\xi_p)_p$, $(\eta_p)_p$ measurable
fields of vectors, there is a unique measurable structure so that
$(\xi_p \otimes \eta_p)_p$ is a measurable field of vectors. We
check
(1) - (3) of \cite[Remark 1.5, (1) - (3)]{Lan}.\\
 (1) $(e_p^n \otimes f_p^m)_p$ is
a $\mu$-measurable field of vectors and $e^n_p \otimes f_p^m \in
\Dom(A_p \otimes B_p)$ for all $p$. The function
$$p \mapsto \langle (A_p \otimes B_p) (e_p^n \otimes f_p^m)_p, (e_p^{n'} \otimes f_p^{m'})_p\rangle = \langle A_p e_p^n, e_p^{n'}\rangle \langle B_p f_p^m, f_p^{m'}\rangle,$$
is $\mu$-measurable, so (2) follows. For (3) fix a $p \in X$.
By definition $\{ e_p^n \mid n\in \mathbb{N}  \}$ is a core for
$A_p$ and $\{ f_p^n \mid n\in \mathbb{N}  \}$ is a core for $B_p$.
Then it follows from \cite[Lemma 11.2.29]{Kad} that
$\textrm{span}_{\mathbb{C}}\left\{ e^n_p \otimes f^m_p \mid n,m
\in \mathbb{N} \right\}$ is a core for $A_p \otimes B_p$, so that
$R$ is total in $\Dom(A_p \otimes B_p)$ with respect to the graph
norm. In all, we have proved (a) and (b).

Using \cite[II.1.3, Remarque 1]{Dix}, we may assume that
$(e^n_p)_p$ (resp. $(f^n_p)_p$) satisfies $p \mapsto
\Vert(e^n_p)_p\Vert$ (resp. $p \mapsto \Vert(f^n_p)_p\Vert$) is
bounded and vanishes outside a set of finite measure. Let
$$\lambda^{n,m}_p = \left( \textrm{max}(1, \Vert (A_p \otimes B_p) (e_p^n \otimes f_p^m) \Vert, \Vert A_p e_p^n \Vert, \Vert, \Vert B_p f_p^m \Vert)\right)^{-1},$$
so $\lambda^{n,m}_p$ is measurable and $0 < \lambda^{n,m}_p \leq 1
$. Using the assumption $\lambda^{n,m}_p (e_p^n \otimes f_p^m)\in
T$. Moreover, $p \mapsto \Vert \lambda_p^{n,m} (e_p^n \otimes
f_p^m)\Vert_{\textrm{Graph}(A_p \otimes B_p)}^2$ is bounded. Let
$S = \{ (\lambda_p^{n,m} (e_p^n \otimes f_p^m))_p \mid n,m \in
\mathbb{N} \} \subseteq T$. Now define
 $$M =  \bigcup_{f \in \mathcal{C}} m_f S, $$
where $\mathcal{C}$ is the set of bounded measurable scalar-valued
functions vanishing outside a set of finite measure and $m_f$ is
multiplication by $f$. Then $M \subseteq T \subseteq
\Dom(\int^{\oplus}_X (A_p \otimes B_p) d\mu(p))$ and by
\cite[II.1.6, Proposition 7]{Dix}, $M$ is total in
$\Dom(\int^{\oplus}_X (A_p \otimes B_p) d\mu(p))$ equipped with the
graph norm. Hence $T$ is a core for $\int^{\oplus}_X (A_p \otimes
B_p) d\mu(p)$. \end{proof}

\section{Appendix}\label{AppendixB}
 
$(M,\Delta)$ denotes the quantum group analogue of the normalizer
of $SU(1,1)$ in $SL(2, \mathbb{C})$. We use the same notation as
in \cite{KoeKus} and \cite{GrKoeKus}. The Casimir operator
$\Omega$ is defined in \cite[Definition 4.5]{GrKoeKus}. $\sigma(\Omega)$ and
$\sigma_d(\Omega)$ denote the spectrum and the discrete spectrum
of $\Omega$ respectively.

\begin{Proposition}\label{PropDiscContCoreps}
Let $x \in [-1,1]$ and $x' \in \sigma_d(\Omega)$, so in particular $x \not= x'$. Then the irreducible summands of $W_{p,x}$ are all inequivalent from $W_{p,x'}$.
\end{Proposition}
\begin{proof} 
This follows from \cite{GrKoeKus}, since the eigenvalues of $\Omega$ when restricted to $W_{p,x} '$ are contained in $\mathbb{R} \backslash [-1,1]$, whereas for $W_{p,x}$ the eigenvalues of $\Omega$ are in $[-1,1]$.

\end{proof}

The next propositions show that $(M, \Delta)$ satisfies the
conditions of the Plancherel theorem, cf. Remark \ref{RmkReducedSetting}.

\begin{Proposition}
 $\hat{M}$ is a type I von Neumann algebra.
\end{Proposition}
\begin{proof} We start with some preliminary remarks. The
projections in $\hat{M}'$ correspond to the invariant subspaces of
$W$ and the minimal projections in $\hat{M}'$ correspond to the
irreducible subspaces of $W$. The partial isometries in $\hat{M}'$
correspond to intertwiners of closed subcorepresentations of $W$.

Let $P \in \hat{M}'$ be the projection on $\bigoplus_{p \in
q^\mathbb{Z}} \int_{[-1,1]}^\oplus \mathcal{L}_{p,x}$. There are
no intertwiners between closed subcorepresentations of
$\bigoplus_{p \in q^\mathbb{Z}} \int_{[-1,1]}^\oplus W_{p,x} dx $
and the direct sum $\bigoplus_{p \in q^\mathbb{Z}} \int_{x \in
\sigma_d(\Omega)}^\oplus W_{p,x} $, see Proposition
\ref{PropDiscContCoreps}. Therefore, $P$ commutes with every
partial isometry in $\hat{M}'$ so that $P$ is central. We have
$\hat{M}' = P\hat{M}'P \oplus  (1-P)\hat{M}'(1-P)$. The von
Neumann algebra $(1-P)\hat{M}'(1-P)$ is of type I since the direct
sum decomposition $\bigoplus_{p \in q^\mathbb{Z}} \int_{x \in
\sigma_d(\Omega)}^\oplus W_{p,x}$ together with the preliminary
remarks yield that every projection majorizes a minimal
projection.

Now we prove that  $P\hat{M}'P$ is a type I von Neumann algebra.
Define the Hilbert spaces
\[
\mathcal{L}_x  = \left( \bigoplus_{p \in q^{\mathbb{Z}}}
\mathcal{L}_{p,x} \right) \oplus \left( \bigoplus_{p \in
q^{\mathbb{Z}}} \mathcal{L}_{p,-x}\right), \quad x \in
(0,1);\qquad \mathcal{L}_0  = \oplus_{p \in q^{\mathbb{Z}}}
\mathcal{L}_{p,0}.
\]
\noindent Then,
\begin{equation}\label{EqnDirectIntDec}
P\mathcal{K}  = \int^\oplus_{[0,1]} \mathcal{L}_x dx,
\end{equation}
\noindent and we let $\mathcal{Z}$ denote the diagonizable
operators with respect to this direct integral decomposition.

We claim that $\mathcal{Z} \subseteq \hat{M}' \subseteq
\mathcal{Z}'$. For the former inclusion, note that the
stepfunctions in $\mathcal{Z}$ are linear combinations of
projections onto invariant subspaces for $\hat{M}$. By the
preliminary remarks we find $\mathcal{Z} \subseteq \hat{M}'$. To
prove that $\hat{M}' \subseteq \mathcal{Z}'$, note that by
\cite[Corollary 4.11]{GrKoeKus}, $\hat{M}'$ is the
{$\sigma$-strong-$\ast$} closure of the linear span of elements
$\hat{J} Q(p_1, p_2, n) \hat{J}$, $p_1, p_2 \in q^\mathbb{Z}$, $n
\in \mathbb{Z}$. The operators $Q(p_1, p_2, n)$ are decomposable
with respect to the direct integral decomposition
(\ref{EqnDirectIntDec}) as was proved in \cite{GrKoeKus}; combine
\cite[Proposition 10.5]{GrKoeKus} together with the direct
integral decomposition \cite[Theorem 5.7]{GrKoeKus} and the
definition of $Q(p_1, p_2, n)$ \cite[Equation (20)]{GrKoeKus}. We
prove that $\hat{J}$ is a decomposable operator with respect to
(\ref{EqnDirectIntDec}). It suffices to show that $\hat{J}
\subseteq \mathcal{Z}'$ \cite[Theorem~II.2.1]{Dix}.

Let $B \subseteq [0,1]$ be a Borel set and let $P_B \in
\mathcal{Z}$ be the operator $P_B = \int^\oplus_{[0,1]} \chi_B(x)
1_{\mathcal{L}_x} dx$, where $\chi_B$ is the indicator function on
$B$. $P_B$ is a projection and we have
\begin{equation}\label{EqnProjectionPB}
\begin{split}
& \chi_{B\cup -B }(\Omega) \mathcal{K}  = \chi_{B\cup -B }(\Omega) \bigoplus_{p,m,\epsilon, \eta} \mathcal{K}(p,m,\epsilon, \eta) =\\
 &  \bigoplus_p \left( \bigoplus_{m,\epsilon\eta = 1} \int^\oplus_{ x \in B \cup -B} \mathbb{C} dx \oplus  \bigoplus_{m,\epsilon\eta = -1} \int^\oplus_{-x \in B \cup -B} \mathbb{C} dx  \right) =\\
&   \bigoplus_p  \bigoplus_{m,\epsilon,\eta} \int^\oplus_{ x \in B \cup -B} \mathbb{C} dx
 =  \bigoplus_p  \int^\oplus_{ x \in B \cup -B} \mathcal{L}_{p,x} dx
 =  \int^\oplus_{ x \in B} \mathcal{L}_{x} dx
 = P_B \mathcal{K},
\end{split}
\end{equation}
\noindent where the second equation uses \cite[Theorem 1.10]{Lan}
and the fact that there is a direct integral decomposition
$\mathcal{K}(p,m,\epsilon, \eta) = \int^\oplus_{\sigma(\Omega)}
\mathbb{C} dx$ such that $\chi_B(\Omega) \mathcal{K}(p,m,\epsilon,
\eta) = \int^\oplus_{\varepsilon \eta x \in B} \mathbb{C} dx$, see
\cite[Theorem 8.13]{GrKoeKus}. Other equations are a matter of
changing the order and combining direct integrals.

%Changing direct integrals is oke, since \int^\oplus \int^oplus \mathbb{C} dx dy = \int^\oplus \mathbb{C} dy \otimes \int^oplus \mathbb{C} dx (zie Dixmier). 2x toepassen geeft verwisselen!

 Note that $\Omega$ leaves the spaces $\mathcal{K}^+$ and $\mathcal{K}^-$ invariant. Let $P^+$ and $P^-$ be the projections onto respectively
$\mathcal{K}^+$ and $\mathcal{K}^-$. Write, again using the notation of \cite{GrKoeKus}
$$\Omega = \left(
\begin{array}{ll}
\Omega^+ & 0  \\
0 & \Omega^-
\end{array}
\right), \qquad \Omega_0 = \left(
\begin{array}{ll}
\Omega^+_0 & 0  \\
0 & \Omega^-_0
\end{array}
\right),
$$
\noindent where $\Omega^\pm = \Omega P^\pm$ and $\Omega_0^\pm =
\Omega_0 P^\pm$. Note that $\Omega^\pm$ is a self-adjoint
extension of $\Omega_0^\pm$. By \cite[Equation (11)]{GrKoeKus} we
see that $\hat{J}$ leaves the spaces $\mathcal{K}^+$ and
$\mathcal{K}^-$ invariant. We claim that
\begin{equation}\label{EqnOperatorOmega}
  \hat{J}\vert_{\mathcal{K}^+} \Omega^+ \hat{J}\vert_{\mathcal{K}^+} = \Omega^+, \qquad   \hat{J}\vert_{\mathcal{K}^-} \Omega^- \hat{J}\vert_{\mathcal{K}^-} = -\Omega^-.
 \end{equation}
\noindent By \cite[Equations (11) and (19)]{GrKoeKus} we find that
$\hat{J} \Omega_0 \hat{J} f_{m,p,t} = \sgn(pt) \Omega_0
f_{m,p,t}$, so that $\hat{J} \Omega_0^+ \hat{J} = \Omega_0^+$ and
$\hat{J} \Omega_0^- \hat{J} = -\Omega_0^-$. Hence $\hat{J}
\Omega^+ \hat{J} \supseteq \Omega_0^+$, and $\hat{J} \Omega^-
\hat{J} \supseteq -\Omega_0^-$. Let $x \in \hat{M}'$, and write:
\[
\begin{split}
\hat{J} x \hat{J} = y^+ \oplus y^-, &\:\:\qquad y^+ = \left(
\begin{array}{ll}
y^{+}_1 & 0  \\
0 & y^{+}_2
\end{array}
\right) \in M_+, \:\: y^- = \left(
\begin{array}{ll}
0 & y^{-}_2  \\
y^{-}_1 & 0
\end{array}
\right) \in M_-,
\end{split}
\]
\noindent where the decomposition is as in \cite[Proposition
4.8]{GrKoeKus}. By that same proposition, we find that $y^-_1
\Omega^+ \subseteq - \Omega^- y^-_1$, $y^-_2 \Omega^- \subseteq -
\Omega^+ y^-_2$, $y^+_1 \Omega^+ \subseteq \Omega^+ y^+_1$ and
$y^+_2 \Omega^- \subseteq \Omega^- y^+_2$. This implies the
inclusion in the following computation:
\[
\begin{split}
& x \hat{J}\left(
\begin{array}{ll}
\Omega^+ & 0  \\
0 & -\Omega^-
\end{array}
\right)\hat{J}   = \hat{J} \hat{J} x \hat{J} \left(
\begin{array}{ll}
\Omega^+ & 0  \\
0 & -\Omega^-
\end{array}
\right) \hat{J} = \\
& \hat{J} \left(
\begin{array}{ll}
y^{+}_1 & 0  \\
0 & y^{+}_2
\end{array}
\right)\left(
\begin{array}{ll}
\Omega^+ & 0  \\
0 & -\Omega^-
\end{array}
\right) \hat{J} \oplus \hat{J} \left(
\begin{array}{ll}
0 & y^{-}_2  \\
y^{-}_1 & 0
\end{array}
\right)\left(
\begin{array}{ll}
\Omega^+ & 0  \\
0 & -\Omega^-
\end{array}
\right) \hat{J} \subseteq \\
&  \hat{J} \left(
\begin{array}{ll}
\Omega^+ & 0  \\
0 & -\Omega^-
\end{array}
\right) (y^+ \oplus y^-) \hat{J}
 = \hat{J} \left(
\begin{array}{ll}
\Omega^+ & 0  \\
0 & -\Omega^-
\end{array}
\right) \hat{J} x.
\end{split}
\]
\noindent So $\hat{J} \Omega^+ \hat{J} \oplus - \hat{J} \Omega^-
\hat{J}$ is a self-adjoint operator affiliated to $\hat{M}$
extending $\Omega_0$. So \cite[Theorem 4.6]{GrKoeKus} implies that
$\left(\hat{J}\vert_{\mathcal{K}^+} \Omega^+ \hat{J}\vert_{\mathcal{K}^+} \oplus - \hat{J}\vert_{\mathcal{K}^-} \Omega^-
\hat{J}\vert_{\mathcal{K}^-} \right) = \Omega$, which results in
(\ref{EqnOperatorOmega}).

\noindent To prove that $\hat{J} \subseteq \mathcal{Z}'$, it
suffices to prove that for all Borel sets $B \subseteq [0,1]$,
$\hat{J} P_B \hat{J} = P_B$. Indeed we have
\[
\begin{split}
 \hat{J} P_B \hat{J} = \hat{J} \chi_{B\cup -B}(\Omega) \hat{J} &=  \hat{J}\vert_{\mathcal{K}^+} \chi_{B \cup -B}(\Omega^+) \hat{J}\vert_{\mathcal{K}^+} \oplus \hat{J}\vert_{\mathcal{K}^-} \chi_{B\cup -B}(\Omega^-) \hat{J}\vert_{\mathcal{K}^-} \\&= \chi_{B\cup-B}(\Omega^+) \oplus \chi_{B\cup -B}(\Omega^-) = \chi_{B\cup -B}(\Omega) =
 P_B.
 \end{split}
\]
\noindent The first and last equality are due to
(\ref{EqnProjectionPB}); the third equality is due to
(\ref{EqnOperatorOmega}). In all, we have proved that $\mathcal{Z}
\subseteq \hat{M} \subseteq \mathcal{Z}'$.

 Let $W_x
= \left( \bigoplus_{p \in q^\mathbb{Z}} W_{p,x}\right) \oplus
\left( \bigoplus_{p \in q^\mathbb{Z}} W_{p,-x}\right)$ for $x \in
(0,1]$ and $W_0 = \bigoplus_{p \in q^\mathbb{Z}} W_{p,0}$. The
operators $Q(p_1,p_2,n)$ form a countable family that generates
$\hat{M}$ \cite[Proposition 4.9]{GrKoeKus}. We apply \cite[Theorem
II.3.2]{Dix} and its subsequent remark, together with
\cite[Theorem II.3.1]{Dix} to conclude that

\begin{equation}\label{EqnMDecomposition}
P\hat{M}P = \int^\oplus_{x \in [0,1]} \hat{M}_x dx,
\end{equation}

\noindent where  $\hat{M}_x$ is generated by $\left\{
(\omega\otimes \iota)(W_x) \mid \omega \in M_\ast  \right\}$
almost everywhere. The projections in $\hat{M}_x'$ correspond to
irreducible subspaces of $W_x$. Since $W_x$ decomposes as a direct
sum of irreducible corepresentations \cite[Proposition 5.4]{GrKoeKus}, every
projection in $\hat{M}_x'$ majorizes a minimal projection. We find
that $\hat{M}_x'$ is type I and by \cite[Theorem 14.1.21]{Kad},
\cite[Corollary V.2.24]{TakI} and (\ref{EqnMDecomposition}) we
conclude that $P\hat{M}P$ is type I . \end{proof}

\begin{Proposition}
 $\hat{M}_c$ is separable.
\end{Proposition}
\begin{proof} Note that if $\omega_n \in M_\ast$ is sequence
that converges in norm to $\omega \in M_\ast$, then $\Vert
\lambda(\omega_n) - \lambda(\omega) \Vert \leq \Vert \omega_n -
\omega \Vert$ so that $\lambda(\omega_n)$ converges in norm to
$\lambda(\omega)$. Since the norm on $\hat{M}_c$ is the operator
norm on the GNS-space and $\hat{M}_c$ is the C*-algebra obtained
as the closure of $\{ \lambda(\omega) \mid \omega \in M_\ast \}$.
It suffices to check that $M_\ast$ is separable. The
$\mathbb{Q}$-linear span of $\{ \omega_{f_{m_0,p_0,t_0},
f_{m_1,p_1,t_1}} \mid m_i \in \mathbb{Z}, p_i, t_i \in I_q, i
=0,1\}$ is weakly dense, hence norm dense in $M_\ast$. \end{proof}

% Your bilbigraphy           %<-------------------

\end{document}